\documentclass[final,leqno,onefignum,onetabnum]{siamltex1213}

\newcommand {\N} {\mathbb N}

\newcommand{\ga}{\gamma}
\newcommand{\la}{\lambda}
\newcommand{\de}{\delta}
\newcommand {\gr} {{\rm gph}}
\newcommand {\gph} {{\rm gph}}
\newcommand{\bx}{\bar x}
\newcommand{\by}{\bar y}
\newcommand{\bz}{\bar z}

\newcommand{\iv}{^{-1} }
\newcommand {\bd} {{\rm bd}\,}
\newcommand {\cone} {{\rm cone}\,}
\newcommand {\Limsup} {\mathop{{\rm Lim\,sup}\,}}
\newcommand {\Liminf} {\mathop{{\rm Lim\,inf}\,}}
\def\usc{upper semicontinuous}

\def\RHS{right-hand side}
\def\SVM{set-valued mapping}

\usepackage{amsmath,amsfonts,amssymb}
\usepackage{color}
\usepackage[nocompress,space]{cite}
\usepackage{exscale}
\renewcommand{\ldots}{\dotsc}
\newcommand{\eps}{\varepsilon}
\newcommand {\R} {\mathbb R}
\newcommand {\dom} {{\rm dom}\,}

\newcommand {\Int} {{\rm int}\,}

\renewcommand {\theenumi} {{\rm (\roman{enumi})}}

\newtheorem{remark}[theorem]{\textit{Remark}}

\title{AN INDUCTION THEOREM AND NONLINEAR REGULARITY MODELS\thanks{The research was supported by the Australian Research Council, project DP110102011 and the National Foundation for Science and Technology Development (NAFOSTED) of Vietnam, project 101.01-2014.62.}}

\author{PHAN Q. KHANH\thanks{Department of Mathematics, International University, Vietnam National University Hochiminh City and Centre for Informatics and Applied Optimization, Faculty of Science and Technology, Federation University Australia (pqkhanh@hcmiu.edu.vn).}
\and ALEXANDER Y. KRUGER\thanks{Centre for Informatics and Applied
Optimization, Faculty of Science and Technology, Federation University Australia (a.kruger@federation.edu.au).}
\and NGUYEN H. THAO\thanks{Centre for Informatics and Applied Optimization, Faculty of Science and Technology, Federation University Australia (hieuthaonguyen@students.federation.edu.au) and
Institute for Numerical and Applied Mathematics, University of G\"ottingen (hieu-thao.nguyen@mathematik.uni-goettingen.de).}
}

\begin{document}
\maketitle
\newcommand{\slugmaster}{%
\slugger{siopt}{xxxx}{xx}{x}{}}

\begin{abstract}
A general nonlinear regularity model for a set-valued mapping $F:X\times\R_+\rightrightarrows Y$, where $X$ and $Y$ are metric spaces, is studied using special iteration procedures, going back to Banach, Schauder, Lyusternik and Graves.
Namely, we revise the \emph{induction theorem} from Khanh, \emph{J. Math. Anal. Appl.}, 118 (1986) and employ it to obtain basic estimates for exploring regularity/openness properties.
We also show that it can serve as a substitution of the Ekeland variational principle when establishing other regularity criteria.
Then, we apply the induction theorem and the mentioned estimates to establish criteria for both global and local versions of regularity/openness properties for our model and demonstrate how the definitions and criteria translate into the conventional setting of a set-valued mapping $F:X\rightrightarrows Y$.
An application to second-order necessary optimality conditions for a nonsmooth set-valued optimization problem with mixed constraints is provided.
\end{abstract}



\pagestyle{myheadings}
\thispagestyle{plain}
\markboth{PHAN Q. KHANH, ALEXANDER Y. KRUGER, NGUYEN H. THAO}{INDUCTION THEOREM AND NONLINEAR REGULARITY MODELS}

\section{Introduction}

Regularity properties of set-valued mappings lie at the core of variational analysis because of their importance for establishing stability of solutions to generalized equations (initiated by Robinson \cite{Rob76,Rob76.2} in the 1970s), optimization and variational problems, constraint qualifications, qualification conditions in coderivative/subdifferential calculus and convergence rates of numerical algorithms; cf. books and surveys \cite{DonRoc14,Mor06.1,AubFra90,BorZhu05, Iof00_,Iof10,Iof13,Pen13,RocWet98,Aze06,KlaKum02} and the references therein.

Among the variety of known regularity properties, the most recognized and widely used one is that of \emph{metric regularity}; cf. \cite{Aze06,Bor86,BorZhu05, Iof00_, KlaKum02, Pen89,Pen13, RocWet98, Mor06.1, DonRoc14}.
Recall that a set-valued mapping $F:X\rightrightarrows{Y}$ between metric spaces is (locally) metrically regular at a point
$(\bar{x},\bar{y})$ in its graph $\gph{F}:=\{(x,y)\in{X}\times{Y}\mid y\in{F}(x)\}$
with modulus $\kappa>0$ if
\begin{equation}\label{met_reg}
d(x,F^{-1}(y))\le\kappa d(y,F(x))
\quad\mbox{for all } x \mbox{ near } \bx,\;y \mbox{ near } \by.
\end{equation}
Here $F\iv:Y\rightrightarrows{X}$ is the \emph{inverse} mapping defined by $F\iv(y)=\{x\in X\mid y\in F(x)\}$.
The roots of this notion can be traced back to the classical Banach-Schauder \emph{open mapping theorem} and its subsequent generalization to nonlinear mappings known as \emph{Lyusternik-Graves theorem}, see the survey \cite{Iof00_} by Ioffe.

Inequality \eqref{met_reg} provides a linear \emph{error bound} estimate of metric type for the distance from $x$ to the solution set of the generalized equation $F(u)\ni y$ corresponding to the perturbed \RHS\ $y$ in a neighbourhood of the solution $\bx$ (corresponding to the \RHS\ $\by$).
Metric regularity is known to be equivalent to two other fundamental properties: the \emph{openness} (or \emph{covering}) at a linear rate and the \emph{Aubin property} (a kind of Lipschitz-like behaviour) of the inverse mapping;  cf. \cite{DmiMilOsm80_, Kru88, Mor06.1,AubEke84, RocWet98,Pen89,Iof00_,DonLewRoc03,DonRoc14, BorZhuang88,Pen13,KlaKum02}.
Several qualitative and quantitative characterizations of the metric regularity property have been established in terms of various primal and dual space derivative-like objects:
\emph{slopes}, \emph{graphical derivatives} (\emph{Aubin criterion}), \emph{subdifferentials} and \emph{code\-rivatives}; cf.~\cite{Iof00_, Kru88, Mor06.1,Aze06,AubFra90,RocWet98,DonRoc14, KlaKum02,Kum99,Pen13,NgaKruThe12}.

There have been many important developments of the metric regularity theory in recent years; among them
clarifying the connection of the metric regularity \emph{modulus} (the infimum of all $\kappa$ such that \eqref{met_reg} holds) to the \emph{radius of metric regularity}, cf.~\cite{DonLewRoc03,DonRoc14,CanGomPar08,Lop12,Mor06.1, GerMorNam09,Ude12.1}, and the interpretation of the regularity of the subdifferential mapping via second-order growth conditions, cf.~\cite{ArtGeo08,DruLew13,DruIof,LewZha13,ZheNg15}.

At the same time, it has been well recognized that many important variational problems do not posses conventional metric regularity.
This observation has led to a significant grows of attention to more subtle regularity properties.
This new development has mostly consisted in the relaxing or extension of the metric regularity property \eqref{met_reg} (and the other two equivalent properties) and its characterizations along the following three directions (and their appropriate combinations).

1) Relaxing of property \eqref{met_reg} by fixing one of the variables: either $y=\by$ or $x=\bx$ in it.
In the first case, one arrives at the very important for applications property of $F$ known as \emph{metric subregularity} (and respectively \emph{calmness} of $F\iv$); cf. \cite{Cla83, DonRoc14,HenJouOut02, HenOut01, Iof00_,IofOut08,ZheNg10,Kru15,Kru15.2,ApeDurStr13,DonRoc04}, while fixing the other variable (and usually also replacing $d(y,F(\bx))$ with $d(y,\by)$) leads to another type of relaxed regularity known as \emph{metric semiregularity} \cite{Kru09.1} (also referred to as \emph{metric hemiregularity} in \cite{ArtMor11}).

2) Considering nonlocal versions of \eqref{met_reg}, when $x$ and $y$ are restricted to certain subsets $U\subset X$ and $V\subset Y$, not necessarily neighbourhoods of $\bx$ and $\by$, respectively, or even a subset $W\subset X\times Y$; cf. \cite{Iof00_,Iof10,Iof11,Iof13}.
A nonlocal regularity (covering) setting was already studied in \cite{DmiMilOsm80_}.

3) Considering nonlinear versions of \eqref{met_reg}, when, instead of the constant modulus $\kappa$, a certain \emph{functional modulus} $\mu:\R_+\to\R_+$ is used in \eqref{met_reg}, i.e., $\kappa d(y,F(x))$ is replaced by $\mu(d(y,F(x)))$; cf. \cite{Iof00_,Iof13,BorZhuang88,Pen89}.
This allows treating more subtle regularity properties arising in applications when the conventional estimates fail.
The majority of researchers focus on the particular case of ``power nonlinearities'' when $\mu$ is of the type $\mu(t)=\lambda t^k$ with $\la>0$ and $0<k\le1$ \cite{Fra87,Fra89,FraQui12,Iof13,YenYaoKie08}.

Starting with Ioffe \cite{Iof79}, most proofs of various sufficient
regularity/open\-ness criteria are based on the application of the celebrated \emph{Ekeland variational principle} (Theorem \ref{Eke}); see \cite{BorZhu05,Mor06.1,Pen13,RocWet98,DonRoc14, Iof00_}.
On the other hand, as observed by Ioffe in \cite{Iof00_}, the original methods used by Banach, Schauder, Lyusternik and Graves had employed special iteration procedures.
This classical approach was very popular in the 1980s -- early 1990s
\cite{Pta82,Kha86,Kha88,Kha89,DmiMilOsm80_,Don96, Com90,Rob80}.
In particular, in the series of three articles \cite{Kha86,Kha88,Kha89},
using iteration techniques several basic statements were established which generalized many known by that time open mapping and closed graph theorems and theorems of the Lyusternik type and results on approximation and semicontinuity or their refinements.
We refer to \cite{Iof00_} for a thorough discussion and comparison of the two main techniques.

In this article, we demonstrate that the approach based on iteration procedures still possesses potential.
In particular, we show that the \emph{Induction theorem} \cite[Theorem~1]{Kha86} (see Lemma~\ref{L1} below), which was used as the main tool when proving the other results in \cite{Kha86}, implies also all the main results in the subsequent articles \cite{Kha88,Kha89}.
It can serve as a substitution of the Ekeland variational principle when establishing other regularity criteria.
Furthermore, the latter classical result can also be established as a consequence of the Induction theorem.
The sequences in the statement of this theorem as well as several other statements in Section~\ref{bas_est} expose the details of iterative procedures which are usually hidden in the proofs of regularity/openness properties.
This is important for the understanding of the roles played by different parameters and leaves some freedom of choice of the parameters defining iteration procedures; this can be helpful when constructing specific schemes as demonstrated in \cite{Kha86,Kha88,Kha89}.

We consider a general nonlocal nonlinear regularity model for a set-valued mapping $F:X\times\R_+\rightrightarrows Y$, where $X$ and $Y$ are metric spaces.
It obviously covers the case of a parametric family of set-valued mappings; cf. \cite{Kha88,Kha89}.
At the same time,
the conventional setting of a set-valued mapping $F:X\rightrightarrows Y$ between metric spaces can be imbedded into the model by defining a set-valued mapping $\mathcal{F}:X\times \R_+\rightrightarrows Y$ by the equality $\mathcal{F}(x,t):=B(F(x),t)=\cup_{y\in F(x)}B(y,t)$ (with the convention $B(y,0)=\{y\}$).
As observed by Ioffe \cite[p.~508]{Iof00_}, this scheme is convenient for deducing regularity/openness estimates.

To define an analogue of metric regularity in this general setting, the distance $d(y,F(x))$ in the image space in the \RHS\ of \eqref{met_reg} is replaced by the ``distance-like'' quantity
\begin{equation}\label{del}
\delta(y,F,x):=\inf\{t>0\mid y\in F(x,t)\}.
\end{equation}
This allows one to define also a natural analogue of the covering property
(but not the Aubin property!)
and establish equivalence of both properties and some sufficient criteria.
If $F(x,t)$ describes the set of positions of a dynamical system feasible at moment $t$ starting at the initial point $x$, then constant \eqref{del} solves the minimal time feasibility problem.


In our study of regularity properties of \SVM s, we follow a three-step procedure which, in our opinion, is important for understanding the roles of particular assumptions employed in the criteria and the origins of specific regularity estimates.

1) Deducing basic regularity estimates at a fixed point $(x,t,y)\in\gph F$.

2) ``Setting free'' variable $t$ in the basic regularity estimates obtained in step 1 and formulating the best (in terms of $t$) estimates.
This way, the ``distance-like'' quantity \eqref{del} comes into play.
The estimates are formulated at a fixed point $(x,y)\in X\times Y$.

3) ``Setting free'' variables $x$ and $y$ in the regularity estimates obtained in step~2, restricting them to a subset $W\subset X\times Y$ and formulating estimates holding for all $(x,y)\in W$.
For the motivations behind such settings we refer the reader to \cite{Iof11,Iof13}.
This way, we arrive at analogues of the metric regularity criteria.
Under the appropriate choice of the set $W$, one can study various local and nonlocal settings of this property and even weaker sub- and semi-regularity versions.
This line goes beyond the scope of the current article.
\medskip

The structure of the article is as follows.
In the next section, we give a short proof of a revised version of the Induction theorem \cite[Theorem~1]{Kha86} and then apply it to establish several basic regularity estimates for a set-valued mapping $F:X\times\R_+\rightrightarrows Y$ at a fixed point $(x,t,y)\in\gph F$.
As a consequence, we obtain the two main theorems from \cite{Kha89} which cover the other results in \cite{Kha86,Kha88}.
Next we discuss the relationship between the Induction theorem and the Ekeland variational principle.
As another consequence of the aforementioned regularity estimates, we deduce several `at a point' sufficient criteria for the regularity of $F$ in terms of quantity \eqref{del}.
Section~\ref{MET_SEC} is devoted to nonlinear \emph{regularity on a set} (and the corresponding openness property) being a direct analogue of metric regularity in the conventional setting.
We refrain from using the term ``metric'' because quantity \eqref{del} is not a distance in the image space.
In Section~\ref{CON}, we demonstrate how the definitions and criteria from Section~\ref{MET_SEC} translate into the conventional setting of a set-valued mapping $F:X\rightrightarrows Y$ taking the natural metric form.
In Section~\ref{S5}, our general nonlinear regularity model is applied to establishing second-order necessary optimality conditions for a general nonsmooth set-valued optimization problem with mixed constraints.
In line with the original idea of Lyusternik, the role of the regularity assumption is to allow handling of the
constraints.
This remains one of the major motivations for the development of the regularity theory.
The final Section~\ref{ConRem} contains some concluding remarks and a list of things to be done hopefully in not-so-distant future.

Our basic notation is standard; cf. \cite{Mor06.1,RocWet98,DonRoc14,Pen13,BorZhu05}.
$X$ and $Y$ are metric spaces.
Metrics in all spaces are denoted by the same symbol
$d(\cdot,\cdot)$.
If $x$ and $C$ are a point and a subset of a metric space, then
$d(x,C):=\inf_{c\in{C}}d(x,c)$ is the point-to-set distance from
$x$ to $C$, while
$\overline{C}$ and $\bd C$ denote the closure and the boundary of $C$.
$B(x,r)$ and $\overline{B}(x,r)$ stand for the open and closed balls of radius $r>0$ centered at $x$, respectively. We use the convention that $B(x,0)=\{x\}$.
If $C$ is a subset of a linear space, then $\cone C:=\{\lambda x\mid  \lambda>0,\; x\in C\}$ is the cone generated by $C$.

\section{Regularity at a point}\label{bas_est}
This section prepares the tools for the study of regularity properties of set-valued mappings in the rest of the article.

\subsection{Basic estimates}\label{BE}

The next technical lemma is a revised version of the \emph{Induction theorem} \cite[Theorem~1]{Kha86}
and contains the core arguments used in the main results of \cite{Kha86,Kha88,Kha89}.
For simplicity, it is formulated for mappings between metric spaces.
(Most of the results in \cite{Kha86,Kha88,Kha89} are formulated in the more general setting of quasimetric spaces.)

Recall that a set-valued mapping $F:X\rightrightarrows Y$ between metric spaces is called outer semicontinuous \cite{RocWet98}
at $\bx\in X$ if
\begin{equation*}
\Limsup_{x\to\bx}F(x):=\{y\in Y\mid \liminf_{x\to\bx}d(y,F(x))=0\}\subset F(\bx).
\end{equation*}

\begin{lemma}\label{L1}
Let $X$ be a complete metric space, $\Phi:\R_+\rightrightarrows X$, $t>0$ and $x\in\Phi(t)$.
Suppose that $\Phi$ is
outer semicontinuous at $0$
and there are sequences of positive numbers $(a_n)$ and $(b_n)$ such that
\begin{gather}\label{A4}
\sum_{n=0}^{\infty}b_n<\infty,
\\\label{A2}
a_0=t
\quad\mbox{and}\quad
a_n\downarrow0\mbox{ as } n\to\infty,
\\\label{A3}
d(u,\Phi(a_{n+1}))< b_n
\quad\mbox{for all}\quad
u\in \Phi(a_n)\cap U_n\quad (n=0,1,\ldots),
\end{gather}
where $U_0:=\{x\}$, $U_n:=B(x,\sum_{i=0}^{n-1}b_{i})$ $(n=1,2,\ldots)$.
Then,
$d(x,\Phi(0))<\sum_{n=0}^{\infty}b_n$.
\end{lemma}
\begin{proof}
Putting $x_0:=x\in\Phi(a_0)\cap U_0$ and using \eqref{A3} repeatedly, we obtain a sequence $(x_n)$ satisfying $x_n\in\Phi(a_n)$ and
$$
d(x_n,x_{n+1})<b_n\quad (n=0,1,\ldots).
$$
The above inequalities together with \eqref{A4} imply that $(x_n)$ is a Cauchy sequence and, as $X$ is complete, converges to some point $z\in X$.
Note that
$$
d(z,x)\le\sum_{n=0}^{\infty} d(x_n,x_{n+1})<\sum_{n=0}^{\infty}b_{n}.
$$
Thanks to the outer semicontinuity of $\Phi$ at $0$ and \eqref{A2}, we have $z\in\Phi(0)$.
Hence, $d(x,\Phi(0))<\sum_{n=0}^{\infty}b_n$.
\end{proof}

\begin{remark}\label{R2.1}
{\rm 1}. The above lemma does not talk about regularity properties of set-valued mappings.
At the same time, we want the reader to observe certain similarity between the conclusion of Lemma~\ref{L1} and inequality \eqref{met_reg} (assuming that $\Phi(0)$ corresponds to the inverse of some set-valued mapping; this is going to be our next step).
The sequences in the statement of the lemma expose  iterative procedures employed in some traditional proofs of regularity properties which can be traced back to Banach and Schauder.

{\rm 2}. As it has been observed by many authors with regards to other regularity statements, with obvious changes, the proof of Lemma~\ref{L1} remains valid if instead of the outer semicontinuity of $\Phi$ and completeness of $X$ one assumes that $\gph\Phi$ is complete (in the product topology).
In fact, it is sufficient to assume that $\gph\Phi\cap (\R_+\times\overline{B}(x,\sum_{n=0}^{\infty}b_{i}))$ is complete.

{\rm 3}. In some applications, a ``restricted'' version of Lemma~\ref{L1} can be useful.
Given a subset $U$ of $X$ and a point $x\in\Phi(t)\cap U$, condition \eqref{A3} can be replaced with the following ``restricted'' one:
\begin{equation*}
d(u,\Phi(a_{n+1})\cap U)< b_n
\quad\mbox{for all}\quad
u\in \Phi(a_n)\cap U_n\quad (n=0,1,\ldots),
\end{equation*}
where $U_0:=\{x\}$, $U_n:=U\cap B(x,\sum_{i=0}^{n-1}b_{i})$ $(n=1,2,\ldots)$.

{\rm 4}. The conclusion of Lemma~\ref{L1} can be equivalently rewritten as
$$
\Phi(0)\cap B\left(x,\sum_{n=0}^{\infty}b_n\right)\ne\emptyset.
$$
\end{remark}

From now on, we consider a set-valued mapping $F:X\times\R_+\rightrightarrows Y$, where $X$ and $Y$ are metric spaces, $X$ is complete.
Given a $t\in\R_+$, we denote $F_t:=F(\cdot,t):X\rightrightarrows Y$.

The purpose of this two-variable model is twofold.
Firstly, if the second variable is interpreted as a parameter, it allows us to cover the case of a parametric family of set-valued mappings; cf. \cite{Kha88,Kha89}.
Secondly, when studying regularity properties of a standard set-valued mapping $F:X\rightrightarrows Y$ between metric spaces, it can sometimes be convenient to consider its two-variable extension $(x,t)\to B(F(x),t)$; cf. Ioffe \cite{Iof00_}.
This model will be explored in Section~\ref{CON}.
In this subsection we focus on the case of a parametric family of set-valued mappings and demonstrate that the main `iterative' results of \cite{Kha88,Kha89} follow easily from Lemma~\ref{L1}.

The next two theorems contain the core arguments of
\cite[Theorems~3 and 4]{Kha89}, respectively.

\begin{theorem}\label{Khanh+}
Let $t>0$ and $(x,t,y)\in\gph F$.
Suppose that the mapping $\tau\mapsto\Phi(\tau):=F_\tau\iv(y)$ on $\R_+$ is outer semicontinuous at $0$
and there are sequences of positive numbers $(b_n)$ and $(c_n)$ and a function $m:(0,\infty)\to(0,\infty)$ such that condition \eqref{A4} holds true and
\begin{gather}\label{B3}
m(\tau)\downarrow0\mbox{ as }\tau\downarrow0
\quad\mbox{and}\quad
c_n\downarrow0\mbox{ as }n\to\infty,
\\\label{B4+}
d(x,F_{m(c_{1})}^{-1}(y))<b_0,
\\\label{B5+}
d(u,F_{m(c_{n+1})}^{-1}(y))< b_n
\;\mbox{for all}\;
u\in F_{m(c_n)}^{-1}(y)\cap B\left(x,\sum_{i=0}^{n-1}b_{i}\right)\; (n=1,2,\ldots).
\end{gather}
Then,
$d(x,F_0\iv(y))<\sum_{n=0}^{\infty}b_n$.
\end{theorem}

\begin{proof}
Set $a_0:=t$, $a_n:=m(c_n)$ $(n=1,2,\ldots)$.
Conditions \eqref{B3}, \eqref{B4+} and \eqref{B5+} imply \eqref{A2} and \eqref{A3}.
By Lemma~\ref{L1}, there exists a $z\in B(x,\sum_{n=0}^{\infty}b_n)$ such that $y\in F(z,0)$, i.e., $z\in F^{-1}_{0}(y)$.
\end{proof}

Given a function $b:\R_+\to\R_+$, we define, for each $t\in\R_+$, $b^0(t):=t$, $b^n(t):=b(b^{n-1}(t))$ $(n=1,2,\ldots)$.

\begin{theorem}\label{Khanh4+}
Let $t>0$ and $(x,t,y)\in\gph F$.
Suppose that the mapping $\tau\mapsto\Phi(\tau):=F_\tau\iv(y)$ on $\R_+$ is outer semicontinuous at $0$ and
there are functions $b,m,\mu:(0,\infty)\to (0,\infty)$ such that
\begin{equation}\label{mutau+}
m(\tau)\downarrow0
\quad\Rightarrow\quad
\tau\downarrow0
\end{equation}
and, for each $\tau>0$ with $\mu(\tau)\le\mu(t)$,
\begin{gather}\label{mu++}
\mu(\tau)\ge m(\tau)+\mu(b(\tau)),
\\\label{net+++}
d(u,F_{b(\tau)}^{-1}(y))<m(\tau) \mbox{ for all } u\in F_{\tau}^{-1}(y)\cap B(x,\mu(t)-\mu(\tau)).
\end{gather}
Then,
$d(x,F_0\iv(y))<\mu(t)$.
\end{theorem}

\begin{proof}
Set $a_n:=b^n(t)$, $b_n:=m(a_n)=m(b^n(t))$ $(n=0,1,\ldots)$.
Adding inequalities \eqref{mu++} corresponding to $\tau=t,b(t),b^2(t),\ldots$, we obtain
\begin{equation*}
\mu(t)\ge \sum_{n=0}^{\infty}m\left(b^n(t)\right) =\sum_{n=0}^{\infty}b_n.
\end{equation*}
Hence, \eqref{A4} is satisfied and $b_n\downarrow 0$ as $n\to \infty$.
Condition \eqref{A2} is satisfied thanks to \eqref{mutau+}.
Condition \eqref{net+++} with $\tau=a_n$ takes the following form:
\begin{equation}\label{tak}
d(u,\Phi(a_{n+1}))<b_n \mbox{ for all } u\in\Phi(a_{n})\cap B(x,\mu(t)-\mu(a_n)).
\end{equation}
For any $n>0$, adding inequalities \eqref{mu++} corresponding to $\tau=t,b(t),\ldots,b^{n-1}(t)$, we obtain
\begin{equation*}
\mu(t)\ge\sum_{i=0}^{n-1}b_{i}+\mu(a_n).
\end{equation*}
Hence, $\mu(a_n)\le\mu(t)$ and condition \eqref{tak} implies \eqref{A3}.
By Lemma~\ref{L1}, there exists a $z\in B(x,\mu(t))$ such that $y\in F(z,0)$.
\end{proof}

\begin{remark}\label{R2.2}
{\rm 1}. The statements of Theorems~\ref{Khanh+} and \ref{Khanh4+} expose the details of iteration procedures which are usually hidden in the proofs of regularity/openness properties.
For instance, the scalar function $b$ in Theorem~\ref{Khanh4+} defines the sequence of iterations corresponding to $\tau\downarrow0$: given a value $\tau$, the next value is $b(\tau)$ which produces a smaller than $\mu(\tau)$ value $\mu(b(\tau))$ of the function $\mu$ with the difference $\mu(\tau)-\mu(b(\tau))$ controlling thanks to \eqref{mu++} the value $m(\tau)$ of the function $m$ which in its turn controls thanks to \eqref{net+++} the distance between the iterations in $X$ leading in the end to the claimed estimate.
Inequalities of the type \eqref{mu++} and \eqref{net+++} (or \eqref{B5+}) are the key ingredients when istablishing regularity estimates.

Theorems~\ref{Khanh+} and \ref{Khanh4+} leave some freedom of choice of the parameters defining iteration procedures which can be helpful when constructing specific schemes as demonstrated in \cite{Kha86,Kha88,Kha89}.

{\rm 2}. Instead of \eqref{B3}, it is sufficient to assume in Theorem~\ref{Khanh+} that $m(c_n)\downarrow0$ as $n\to \infty$.
In Theorem~\ref{Khanh4+}, this is satisfied automatically thanks to \eqref{mutau+}.

{\rm 3}. The conclusions of Theorems~\ref{Khanh+} and \ref{Khanh4+} can be equivalently rewritten as
$
y\in F\left(B\left(x,r\right),0\right)
$
where either $r=\sum_{n=0}^{\infty}b_n$ or $r=\mu(t)$.
\end{remark}

Theorem~\ref{Khanh4+} covers a seemingly more general setting of regularity/covering on a system of balls; cf. \cite{DmiMilOsm80_,Iof00_,Kha89}.

Recall that a family $\Sigma$ of balls in $X$ is called a \emph{complete system} \cite[Definition~1.1]{DmiMilOsm80_} if, for any $B(x,r)\in\Sigma$, one has $B(x',r')\in\Sigma$ provided that $x'\in X$, $r'>0$ and $d(x,x')+r'\le r$.
For a subset $M$ of $X$, $\Sigma(M)$ denotes a complete system of balls $B(x,r)$ in $X$ with $B(x,r)\subset M$.
Obviously the family of all balls in $X$ forms a complete system.

\begin{corollary}\label{Khanh4+.1}
Let $M\subset X$ and $\Sigma(M)$ be a complete system, $t>0$ and $(x,t,y)\in\gph F$.
Suppose that the mapping $\tau\mapsto F_\tau\iv(y)$ on $\R_+$ is outer semicontinuous at $0$ and
there are functions $b,m,\mu:(0,\infty)\to (0,\infty)$ such that $B(x,\mu(t))\in\Sigma(M)$, condition \eqref{mutau+} is satisfied
and, for each $\tau>0$ with $\mu(\tau)\le\mu(t)$, condition \eqref{mu++} holds true and
\begin{equation}\label{net++}
d(u,F_{b(\tau)}^{-1}(y))<m(\tau) \mbox{ for all } u\in F_{\tau}^{-1}(y)\cap \{x'\mid B(x',\mu(\tau))\in \Sigma(M)\}.
\end{equation}
Then,
$d(x,F_0\iv(y))<\mu(t)$.
\end{corollary}

\begin{proof}
Since $B(x,\mu(t))\in\Sigma(M)$, it follows that
$B(x,\mu(t)-\mu(\tau))
\subset\{x'\mid B(x',\mu(\tau))\in \Sigma(M)\}$.
The conclusion follows from Theorem~\ref{Khanh4+}.
\end{proof}

The key estimates \eqref{net+++} and \eqref{net++} in Theorem~\ref{Khanh4+} and Corollary~\ref{Khanh4+.1} are for the original space $X$.
In some situations, one can use for that purpose also similar estimates in the image space $Y$.

\begin{corollary}\label{Khanh4+.2}
Let $t>0$ and $(x,t,y)\in\gph F$.
Suppose that the mapping $\tau\mapsto F_\tau\iv(y)$ on $\R_+$ is outer semicontinuous at $0$ and
there are functions $b,m,\mu:(0,\infty)\to (0,\infty)$ such that condition \eqref{mutau+} is satisfied
and, for each $\tau>0$ with $\mu(\tau)\le\mu(t)$, condition \eqref{mu++} holds true and
\begin{gather}\label{set1}
F_0\iv(B(y,\tau))\subset F_\tau\iv(y),
\\\label{set2}
d\left(y,F_0(B(u,m(\tau)))\right)<b(\tau) \mbox{ for all } u\in F_{\tau}^{-1}(y)\cap B(x,\mu(t)-\mu(\tau)).
\end{gather}
Then,
$d(x,F_0\iv(y))<\mu(t)$.
\end{corollary}
\begin{proof}
Observe that conditions \eqref{set1} and \eqref{set2} imply \eqref{net+++}.
Indeed, if $u\in F^{-1}_\tau(y)\cap B(x,\mu(t)-\mu(\tau))$, then, by \eqref{set2}, there exists a $z\in B(u,m(\tau))$ such that $d(y,F_0(z))<b(\tau)$, or equivalently, $z\in F_0\iv(B(y,b(\tau)))$.
It follows from \eqref{set1} that $z\in F\iv_{b(\tau)}(y)$.
Hence, $d(u,F\iv_{b(\tau)}(y))<m(\tau)$.
The conclusion follows from Theorem~\ref{Khanh4+}.
\end{proof}

\begin{remark}\label{ha}
{\rm 1}. Instead of \eqref{mutau+}, it is sufficient to assume in Theorem~\ref{Khanh4+} and Corollaries~\ref{Khanh4+.1} and \ref{Khanh4+.2} that $b^n(t)\downarrow0$ as $n\to \infty$.
The last condition is satisfied, e.g., when $b(t)=\lambda t$ with $\lambda\in(0,1)$.

{\rm 2}. If condition \eqref{mu++} holds true for all $\tau>0$ with $\mu(\tau)\le\mu(t)$, then $\mu(\tau)\ge \sum_{n=0}^{\infty}m\left(b^n(\tau)\right)$.
On the other hand, if the last condition holds true as equality (for all $\tau>0$ with $\mu(\tau)\le\mu(t)$), then condition \eqref{mu++} is satisfied (as equality).
Hence, condition \eqref{mu++} in Theorem~\ref{Khanh4+} and Corollaries~\ref{Khanh4+.1} and \ref{Khanh4+.2} can be replaced by the following definition of the smallest function $\mu$ satisfying \eqref{mu++}:
\begin{equation}\label{mu+}
\mu(\tau):= \sum_{n=0}^{\infty}m\left(b^n(\tau)\right),
\end{equation}
thus producing the strongest conclusion.

{\rm 3}. It is sufficient to assume in Theorem~\ref{Khanh4+} and Corollaries~\ref{Khanh4+.1} and \ref{Khanh4+.2} that conditions \eqref{mu++}, \eqref{net+++}, \eqref{net++}, \eqref{set1} and \eqref{set2} are satisfied only for $\tau=t,b(t),b^2(t),\ldots$
In particular, if this sequence is monotone (as in the typical example mentioned in part 1 above or, thanks to \eqref{mu++} when $\mu$ is nondecreasing), then the conclusions of all the statements remain true when conditions \eqref{mu++}, \eqref{net+++}, \eqref{net++}, \eqref{set1} and \eqref{set2} are satisfied for all $\tau\in(0,t]$.

{\rm 4}. Thanks to part 3, instead of conditions \eqref{net+++}, \eqref{net++} and \eqref{set2}, one can require that, for each $n=0,1,\ldots$, the following conditions hold true, respectively:
\begin{align}\label{r3.1}
d(u,F_{b^{n+1}(t)}^{-1}(y))<m(b^n(t)) \mbox{ for all } u&\in F_{b^n(t)}^{-1}(y)
\\&\notag
\cap B(x,\mu(t)-\mu(b^n(t))),
\\\notag
d(u,F_{b^{n+1}(t)}^{-1}(y))<m(b^n(t)) \mbox{ for all } u&\in F_{b^n(t)}^{-1}(y)
\\&\notag
\cap \{x'\mid B(x',\mu(b^n(t)))\in \Sigma(M)\},
\\\label{r3.3}
d\left(y,F_0(B(u,m(b^n(t))))\right)<b^{n+1}(t) \mbox{ for all } u&\in F_{b^n(t)}^{-1}(y)
\\&\notag
\cap B(x,\mu(t)-\mu(b^n(t))).
\end{align}
If $\mu$ is given by \eqref{mu+}, then conditions \eqref{r3.1} and \eqref{r3.3} can be equivalently rewritten as follows:
\begin{align*}
d(u,F_{b^{n+1}(t)}^{-1}(y))<m(b^n(t)) \mbox{ for all } u&\in F_{b^n(t)}^{-1}(y)
\cap B(x,\sum_{i=0}^{n-1}b^i(t)),
\\
d\left(y,F_0(B(u,m(b^n(t))))\right)<b^{n+1}(t) \mbox{ for all } u&\in F_{b^n(t)}^{-1}(y)
\cap B(x,\sum_{i=0}^{n-1}b^i(t)).
\end{align*}

{\rm 5}. The conclusions of Theorem~\ref{Khanh4+} and Corollaries~\ref{Khanh4+.1} and \ref{Khanh4+.2} can be equivalently rewritten as $y\in F(B(x,\mu(t)),0)$.
\end{remark}

The next two theorems are the (slightly improved) original results of \cite[Theorems 3 and 4]{Kha89} reformulated in the setting of metric spaces and adopting the terminology and notation of the current article.
These theorems, which follow immediately from Theorems~\ref{Khanh+} and \ref{Khanh4+}, respectively, imply all the other results of
\cite{Kha86,Kha88,Kha89} as well as many open mapping and closed graph theorems and theorems of the Lyusternik type and results on approximation and semicontinuity
or their refinements; cf. the references in \cite{Kha86,Kha88,Kha89}.

\begin{theorem}\label{Khanh}
Let $t>0$ and $(x,t)\in\dom F$.
Suppose that, for each $y\in Y$,
\begin{equation}\label{B1}
F_0\iv(y)=\Limsup_{t\downarrow0}F_t\iv(y)
\end{equation}
and there are positive numbers $\rho$, $s$ and $b_n$ $(n=1,2,\ldots)$, such that
\begin{equation}\label{B2}
\sum_{n=1}^{\infty}b_n+s\le\rho.
\end{equation}
Suppose also that, for each $y\in F(x,t)$, there are numbers $c_n>0$ $(n=1,2,\ldots)$ and a function $m:(0,\infty)\to(0,\infty)$ satisfying \eqref{B3} and
\begin{gather}\label{B4}
d(u,F_{m(c_{1})}^{-1}(y))< s
\quad\mbox{for all}\quad
u\in F_{t}^{-1}(y)\cap B(x,\rho-s),
\\\label{B5}
d(u,F_{m(c_{n+1})}^{-1}(y))< b_n
\;\;\mbox{for all}\;\;
u\in F_{m(c_n)}^{-1}(y)\cap B(x,\rho-b_{n})\; (n=1,2,\ldots)\hspace{-.1cm}
\end{gather}
Then, $F(x,t) \subset F(B(x,\rho),0)$.
\end{theorem}

\begin{proof}
Set $b_0:=s$ and
take any $y\in F(x,t)$.
It follows from \eqref{B1} that the mapping $\tau\mapsto F_\tau\iv(y)$ on $\R_+$ is outer semicontinuous at $0$.
Condition \eqref{B2} obviously implies \eqref{A4}.
Observe that $\sum_{i=0}^{n-1}b_{i}\le \rho-\sum_{i=n}^\infty b_{i}<\rho-b_n$ $(n=0,1,\ldots)$
Hence, conditions \eqref{B4} and \eqref{B5} imply \eqref{B4+} and \eqref{B5+}, respectively.
By Theorem~\ref{Khanh+}, $y\in F(B(x,\rho),0)$.
\end{proof}

\begin{theorem}\label{Khanh4}
Let $M\subset X$ and $\Sigma(M)$ be a complete system.
Let a function $b:(0,\infty)\to (0,\infty)$ be given.
Suppose that, for each $y\in Y$, condition \eqref{B1} holds true and there exists a function $m:(0,\infty)\to(0,\infty)$ satisfying condition \eqref{mutau+} and, for all $\tau\in(0,\infty)$ and $x\in X$ with $(x,t,y)\in\gph F$ and $B(x,\mu(\tau))\in\Sigma(M)$, conditions \eqref{net++} and \eqref{mu+} are satisfied.
Then, for any $(x,t,y)\in\gph F$ with $t>0$ and $B(x,\mu(t))\in\Sigma(M)$, one has
$y\in F(B(x,\mu(t)),0)$.
\end{theorem}

\begin{proof}
Take any $(x,t,y)\in\gph F$ with $t>0$ and $B(x,\mu(t))\in\Sigma(M)$ and a function $m$ satisfying the assumptions of the theorem.
Condition \eqref{B1} obviously implies that the mapping $\tau\mapsto F_\tau\iv(y)$ on $\R_+$ is outer semicontinuous at $0$.
Thanks to Remark~\ref{ha}.2, all the assumptions of Corollary~\ref{Khanh4+.1} are satisfied.
Hence, $y\in F(B(x,\mu(t)),0)$.
\end{proof}

\begin{remark}
Comparing the statements of Theorem~\ref{Khanh4} and \cite[Theorem~4]{Kha89}, one can notice that the latter one looks stronger: it is formulated without assumption \eqref{mutau+} and with the stronger conclusion $F(x,t) \subset F(B(x,\mu(t)),0)$.
However assumption \eqref{mutau+} is implicitly used in the proof of \cite[Theorem~4]{Kha89} and the conclusion is established for a fixed $y\in F(x,t)$ satisfying $B(x,\mu(t))\in\Sigma(M)$.
(Observe that function $m$ in Theorem~\ref{Khanh4} and consequently function $\mu$ defined by \eqref{mu+} depend on the choice of $y\in F(x,t)$.)

Unlike the setting of the current article, in \cite{Kha89} mapping $F$ was assumed to be defined not on $X\times\R_+$, but on $X\times[0,t_0]$ where $t_0$ is a given positive number.
This difference can be easily eliminated by setting $F(x,t):=\emptyset$ when $t>t_0$ and making appropriate minor changes in the statements.
\end{remark}

\subsection{Lemma~\ref{L1} and Ekeland variational principle}

Lemma~\ref{L1} which lies at the core of the proofs of the statements in the previous subsection can serve as a substitution of the Ekeland variational principle which is a traditional tool when establishing regularity criteria.
This is demonstrated by the proof of such a criterion in the following theorem.

\begin{theorem}\label{T2.3}
Let $t>0$ and $(x,t,y)\in\gph F$.
Suppose that the mapping $\tau\mapsto F_\tau\iv(y)$ is outer semicontinuous on $[0,t)$
and there is a continuous nondecreasing function $\mu:[0,t]\to \R_+$
satisfying $\mu(\tau)=0$ if and only if $\tau=0$ and,
for each pair $(u,\tau)\in F\iv(y)$ with $\tau\in(0,t]$ and $d(x,u)\le \mu(t)-\mu(\tau)$,
there exists a pair $(u',\tau')\in F\iv(y)$ such that $u'\ne u$ and
\begin{equation}\label{mu-}
\mu(\tau')\le\mu(\tau)-d(u',u).
\end{equation}
Then,
$d(x,F_0\iv(y))\le\mu(t)$.
\end{theorem}
\begin{proof}
Set $a_0:=t$, $\bx:=x$ and define a sequence $\{(x_n,a_n)\}$ by induction.
For any $n=0,1,\ldots$, let a pair $(x_n,a_n)\in F\iv(y)$ with $a_n\in[0,t]$ and $d(x,x_n)\le \mu(t)-\mu(a_n)$ be given.
If $a_n=0$, set $a_{n+1}:=0$ and $x_{n+1}:=x_n$.
Otherwise, define
\begin{equation}\label{T2.3.3}
c_n:=\inf\{\mu(\tau)\mid (u,\tau)\in F\iv(y),\;
\mu(\tau)\le\mu(a_n)-d(u,x_n)\}.
\end{equation}
By the assumptions of the theorem, $0\le c_n<\mu(a_n)$, and one can choose a pair $(x_{n+1},a_{n+1})\in F\iv(y)$ such that $x_{n+1}\ne x_n$ and
\begin{gather}\label{T2.3.1}
\mu(a_{n+1})\le\mu(a_n)-d(x_n,x_{n+1}),
\\\label{T2.3.2}
c_n\le\mu(a_{n+1})<\frac{\mu(a_n)+c_n}{2}<\mu(a_n).
\end{gather}
It also follows from \eqref{T2.3.1} that
$$
d(x,x_{n+1})\le d(x,x_{n})+d(x_n,x_{n+1})\le \mu(t)-\mu(a_{n+1}).
$$

If $a_n=0$ for some $n>0$, then, by \eqref{T2.3.1},
\begin{equation*}
d(x,F_0\iv(y))\le d(x,x_n)\le\sum_{j=0}^{n-1}d(x_j,x_{j+1})\le\mu(t).
\end{equation*}

Now assume that $a_n>0$ for all $n=0,1,\ldots$.
Then, $\{a_n\}$ is a decreasing sequence of positive numbers which converges to some $a\ge0$.
We are going to show that $a=0$.
Suppose that $a>0$ and denote $\hat a_n:=a_n-a$.
Obviously, $\hat a_n>0$ and $\hat a_n\downarrow0$.
By \eqref{T2.3.1},
\begin{equation*}
\sum_{n=0}^{\infty}d(x_n,x_{n+1}) \le\mu(t)-\mu(a).
\end{equation*}
Fix an $\eps>0$ and choose numbers $b_n>d(x_n,x_{n+1})$ such that $\sum_{n=0}^{\infty}b_n< \mu(t)-\mu(a)+\eps$.
Set $\Phi(\hat a_n):=\{x_n\}$, $\Phi(\tau):=\emptyset$ for any $\tau\in(0,\infty) \setminus\{\hat a_0,\hat a_1,\ldots\}$, and let $\Phi(0)$ be the set of all cluster points of $\{x_n\}$.
Then, $x\in\Phi(\hat a_0)$, $\Phi$ is outer semicontinuous at $0$ and $d(\Phi(\hat a_n),\Phi(\hat a_{n+1}))<b_n$.
It follows from Lemma~\ref{L1} that there exists a $z\in \Phi(0)$ satisfying $d(x,z)< \mu(t)-\mu(a)+\eps$.
By the outer semicontinuity of $\Phi$, $y\in F(z,a)$.

Since $a>0$, by the assumptions of the theorem, there exists a pair $(u,\tau)\in F\iv(y)$ such that $u\ne z$ and
\begin{equation}\label{T2.3.4}
\mu(\tau)\le\mu(a)-d(u,z).
\end{equation}
Then, $\mu(\tau)<\mu(a)$.
Observe from \eqref{T2.3.2} that
\begin{equation*}
2\mu(a_{n+1})-\mu(a_n)<c_n<\mu(a_n).
\end{equation*}
Hence, $\{c_n\}$ converges to $\mu(a)$ and consequently $\mu(\tau)<c_n$ when $n$ is large enough.
By definition \eqref{T2.3.3}, this yields
\begin{equation}\label{T2.3.5}
\mu(\tau)>\mu(a_n)-d(u,x_n).
\end{equation}
At the same time,
\begin{equation*}
d(x_n,z)\le \sum_{j=n}^{\infty}d(x_j,x_{j+1}) \le\mu(a_n)-\mu(a).
\end{equation*}
This combined with \eqref{T2.3.4} gives
\begin{equation*}
\mu(\tau)\le\mu(a_n)-d(u,x_n)
\end{equation*}
which is in obvious contradiction with \eqref{T2.3.5}.
Hence, $a=0$, $z\in F_0\iv(y)$, $d(x,z)< \mu(t)+\eps$, and, as $\eps$ is arbitrary, $d(x,F_0\iv(y))\le\mu(t)$.
\end{proof}

The proof of Theorem~\ref{T2.3} given above relies on Lemma~\ref{L1} and uses standard arguments typical for traditional proofs of the Ekeland variational principle; cf. e.g. \cite{BorZhu05}.
We next show that the latter classical result can also be established as a consequence of Lemma~\ref{L1}.

\begin{theorem}[Ekeland variational principle] \label{Eke}
Let $X$ be a complete metric space and $f: X\to\R\cup\{+\infty\}$ be lower semicontinuous and bounded from below.
Suppose $\varepsilon>0$, $\lambda>0$ and $x\in X$ satisfies
$$
f(x)<\inf_X f + \varepsilon.
$$
Then, there exists a $z\in X$ such that
\begin{enumerate}
\item[\rm(i)]
$d(z,x)<\lambda $,
\item[\rm(ii)]
$f(z)\le f(x)$,
\item[\rm(iii)]
$f(u)+(\varepsilon/\lambda)d(u,z)\ge f(z)$ for all $u\in X$.
\end{enumerate}
\end{theorem}

\begin{proof}
Denote $\bx:=x$.
For $n=0,1,\ldots$, set
\begin{equation}\label{Eke1}
a_n:=\sup_{u\in X}\left\{f(x_n)-f(u)-\frac{\eps}{\la}d(u,x_n)\right\}.
\end{equation}
Obviously, $0\le a_n<\infty$.
Choose an $x_{n+1}$ such that
\begin{equation}\label{Eke2}
f(x_n)-f(x_{n+1})- \frac{\eps}{\la}d(x_{n+1},x_n) \ge\frac{a_n}{2}.
\end{equation}
Then, for $n=0,1,\ldots$,
$$
f(x_{n+1})\le f(x_n),
\quad
d(x_{n+1},x_n)\le \frac{\la}{\eps}(f(x_n)-f(x_{n+1}))
$$
and the inequalities are strict if $a_n>0$.
It follows that
$$
f(x_{n})\le f(x)
\quad\mbox{and}\quad
d(x_{n},x)\le \frac{\la}{\eps}(f(x)-f(x_{n}))<\la.
$$
If, for some $n$, $a_n=0$, then $z:=x_n$ satisfies the conclusions of the theorem.
Suppose that $a_n>0$ for all $n=0,1,\ldots$.
Then, $b_n:=\frac{\la}{\eps}(f(x_n)-f(x_{n+1}))>0$.
Set $\Phi(a_n):=\{x_n\}$, $\Phi(\tau):=\emptyset$ for any $\tau\in(0,\infty) \setminus\{a_0,a_1,\ldots\}$ and $\Phi(0):= \Limsup_{\tau\downarrow0} \Phi(\tau)$.
Hence, $\Phi$ is outer semicontinuous at $0$, $x\in\Phi(a_0)$, $\sum_{n=0}^{\infty}b_n<\la$ and $d(\Phi(a_n),\Phi(a_{n+1}))<b_n$.
Besides, it follows from \eqref{Eke1} that
\begin{equation}\label{Eke3}
f(x_{n})-f(u)-\frac{\eps}{\la}d(u,x_{n})\le a_n
\quad\mbox{for all}\quad
u\in X.
\end{equation}
Subtracting \eqref{Eke2} from the last inequality and using the triangle inequality, we conclude that
\begin{equation*}
f(x_{n+1})-f(u)-\frac{\eps}{\la}d(u,x_{n+1})\le \frac{a_n}{2}
\quad\mbox{for all}\quad u\in X,
\end{equation*}
i.e., $a_{n+1}\le a_n/2$ and consequently $a_n\downarrow0$ as $n\to\infty$.
It follows from Lemma~\ref{L1} that there exists a $z\in \Phi(0)$ satisfying (i).
By the definition of $\Phi(0)$ and \eqref{Eke3}, we conclude that conditions (ii) and (iii) are satisfied too.
\end{proof}

\begin{remark}
Lemma \ref{L1} was used in the proof of Theorem~\ref{Eke} where one would normally use the convergence of a Cauchy sequence.
Similarly, the Ekeland variational principle can replace the Cauchy sequence argument in the proof of Lemma \ref{L1}.
In fact, both Lemma \ref{L1} and Theorem~\ref{Eke} are in a sense equivalent to the completeness of $X$.
\end{remark}

\subsection{Regularity}

Lemma \ref{L1} and the other results in Subsection~\ref{BE} provide a collection of basic estimates which are going to be used when establishing regularity criteria.
Theorems~\ref{Khanh+}, \ref{Khanh4+} and \ref{T2.3} and Corollaries~\ref{Khanh4+.1} and \ref{Khanh4+.2} were formulated for a fixed point $(x,t,y)\in\gph F$.
The next step is to ``set variable $t$ free'' and formulate criteria for a fixed point $(x,y)$ such that $(x,t,y)\in\gph F$ for some $t>0$.
Once variable $t$ is free, it is natural to take infimum over $t$ in the \RHS s of the inequalities in the conclusions of the statements in Subsection~\ref{BE} to obtain the best possible estimates.
Under the natural assumption of monotonicity of the function $\mu$ involved in most of the statements, this is equivalent to evaluating the infimum of $t>0$ such that $(x,t,y)\in\gph F$.
This way the ``distance-like'' quantity $\delta(y,F,x)$ defined by \eqref{del} comes into play.

The next several assertions are immediate consequences of Theorems~\ref{Khanh+}, \ref{Khanh4+} and \ref{T2.3} and Corollaries~\ref{Khanh4+.1} and \ref{Khanh4+.2}, respectively.
As an illustration, we provide a short proof of the first one.

\begin{theorem}\label{Kh}
Let $(x,y)\in X\times Y$ and $\mu:\R_+\to\R_+$ be an \usc\ nondecreasing function.
Suppose that the mapping $\tau\mapsto F_\tau\iv(y)$ on $\R_+$ is outer semicontinuous at $0$ and, for some $\ga>\delta(y,F,x)$ and any $t\in(0,\ga)$ with $(x,t,y)\in\gph F$,
there are sequences of positive numbers $(b_n)$ and $(c_n)$ and a function $m:(0,\infty)\to(0,\infty)$ such that conditions \eqref{B3}--\eqref{B5+} hold true and
\begin{equation}\label{A4-}
\sum_{n=0}^{\infty}b_n\le\mu(t).
\end{equation}
Then,
$d(x,F_0\iv(y))\le\mu(\delta(y,F,x))$.
\end{theorem}

\begin{proof}
It is sufficient to notice that, for any $t\in(0,\ga)$ with $(x,t,y)\in\gph F$, condition \eqref{A4-} implies \eqref{A4} and, by Theorem~\ref{Khanh+},
$d(x,F_0\iv(y))<\mu(t)$.
Taking the infimum in the \RHS\ of the above inequality
over all $t>0$ with $(x,t,y)\in\gph F$ and making use of the monotonicity of $\mu$, we arrive at the claimed conclusion.
\end{proof}

\begin{theorem}\label{Kh4}
Let $(x,y)\in X\times Y$ and $\mu:\R_+\to\R_+$ be an \usc\ nondecreasing function.
Suppose that the mapping $\tau\mapsto F_\tau\iv(y)$ on $\R_+$ is outer semicontinuous at $0$ and, for some $\ga>\delta(y,F,x)$ and any $t\in(0,\ga)$ with $(x,t,y)\in\gph F$,
there are functions $b,m:(0,\infty)\to (0,\infty)$ such that condition \eqref{mutau+} is satisfied
and, for each $\tau>0$ with $\mu(\tau)\le \mu(t)$, conditions \eqref{mu++} and \eqref{net+++} hold true.
Then,
$d(x,F_0\iv(y))\le\mu(\delta(y,F,x))$.
\end{theorem}

\begin{corollary}\label{Kh4.1}
Let $M\subset X$ and $\Sigma(M)$ be a complete system, $(x,y)\in X\times Y$ and $\mu:\R_+\to\R_+$ be an \usc\ nondecreasing function.
Suppose that the mapping $\tau\mapsto F_\tau\iv(y)$ on $\R_+$ is outer semicontinuous at $0$ and, for some $\ga>\delta(y,F,x)$ and any $t\in(0,\ga)$ with $(x,t,y)\in\gph F$, one has $B(x,\mu(t))\in\Sigma(M)$,
there are functions $b,m:(0,\infty)\to (0,\infty)$ such that condition \eqref{mutau+} is satisfied
and, for each $\tau>0$, conditions \eqref{mu++} and \eqref{net++} hold true.
Then,
$d(x,F_0\iv(y))\le\mu(\delta(y,F,x))$.
\end{corollary}

\begin{corollary}\label{Kh4.2}
Let $(x,y)\in X\times Y$ and $\mu:\R_+\to\R_+$ be an \usc\ nondecreasing function.
Suppose that the mapping $\tau\mapsto F_\tau\iv(y)$ on $\R_+$ is outer semicontinuous at $0$ and, for some $\ga>\delta(y,F,x)$ and any $t\in(0,\ga)$ with $(x,t,y)\in\gph F$,
there are functions $b,m:(0,\infty)\to (0,\infty)$ such that condition \eqref{mutau+} is satisfied
and, for each $\tau>0$ with $\mu(\tau)\le \mu(t)$, conditions \eqref{mu++}, \eqref{set1} and \eqref{set2} hold true.
Then,
$d(x,F_0\iv(y))\le\mu(\delta(y,F,x))$.
\end{corollary}

\begin{remark}\label{R2.4}
Most of the comments in Remarks~\ref{R2.2} and \ref{ha} are applicable to Theorems~\ref{Kh} and \ref{Kh4} and Corollaries~\ref{Kh4.1} and \ref{Kh4.2}.
\end{remark}

In the next theorem, we at last get rid of the technical parameters inherited from the statements in Subsection~\ref{BE} and formulate a regularity statement in a more conventional way (though still as an ``at a point" condition).

\begin{theorem}\label{T2.6}
Let $(x,y)\in X\times Y$, $\mu:\R_+\to\R_+$ be a continuous nondecreasing function and $\mu(\tau)=0$ if and only if $\tau=0$.
Suppose that the mapping $\tau\mapsto F_\tau\iv(y)$ is outer semicontinuous on $[0,\delta(y,F,x)]$ and, for each pair $(u,\tau)\in F\iv(y)$ with $\tau\in(0,\delta(y,F,x)]$ and $d(x,u)\le \mu(\delta(y,F,x))-\mu(\delta(y,F,u))$, there exists a pair $(u',\tau')\in F\iv(y)$ such that $u'\ne u$ and condition \eqref{mu-} is satisfied.
Then,
$d(x,F_0\iv(y))\le\mu(\delta(y,F,x))$.
\end{theorem}
\begin{proof}
If $\delta(y,F,x)=\infty$, then the conclusion holds true trivially.
Otherwise, the outer semicontinuity of $\tau\mapsto F_\tau\iv(y)$ ensures that
$y\in F(x,\delta(y,F,x))$, and the conclusion follows from Theorem \ref{T2.3} for $t=\delta(y,F,x)$.
\end{proof}

\begin{remark}
The conclusion of Theorems~\ref{Kh}, \ref{Kh4} and \ref{T2.6} and Corollaries~\ref{Kh4.1} and \ref{Kh4.2} reminds the inequality in the definition of the metric regularity property for a set-valued mapping $F:X\rightrightarrows Y$ between metric spaces; cf. \cite{DonRoc14}.
The difference is in the \RHS, where $\delta(y,F,x)$ stands in place of $d(y,F(x))$.
The relationship between the two settings will be explored in Section~\ref{CON}.
\end{remark}

The conclusion of Theorems~\ref{Kh}, \ref{Kh4} and \ref{T2.6} and Corollaries~\ref{Kh4.1} and \ref{Kh4.2} can be reformulated
equivalently in a ``covering-like'' form.

\begin{proposition}\label{P2.1}
Consider the following conditions:
\begin{enumerate}
\item[\rm(i)]
$d(x,F_0\iv(y))\le\mu(\delta(y,F,x))$,
\item[\rm(ii)]
$y\in F(B(x,t),0)$
for any
$t>\mu(\delta(y,F,x))$,
\item[\rm(iii)]
$y\in F(B(x,\mu(\delta(y,F,x))),0).$
\end{enumerate}
Then, {\rm (iii) $\Rightarrow$ (ii) $\Leftrightarrow$ (i)}.
\end{proposition}

\begin{proof}
{\rm (iii) $\Rightarrow$ (ii)} is obvious.

{\rm (i) $\Rightarrow$ (ii)}.
By (i), for any $t>\mu(\delta(y,F,x))$, there exists a $z\in F_0\iv(y)$ such that $d(x,z)<t$ and consequently $y\in F(z,0)\subset F(B(x,t),0)$.

{\rm (ii) $\Rightarrow$ (i)}.
$y\in F(B(x,t),0)$ and $t>0$ if and only if $d(x,F_0\iv(y))<t$.
If the last inequality holds for all $t>\mu(\delta(y,F,x))$, then $d(x,F_0\iv(y))\le\mu(\delta(y,F,x))$.
\end{proof}

\begin{remark}\label{R2.5}
Proposition~\ref{P2.1} is true without the assumption of the completeness of $X$.
\end{remark}

\section{Regularity on a set}\label{MET_SEC}

In this section, we continue exploring
regularity properties for
a set-valued mapping $F:X\times\R_+\rightrightarrows Y$, where $X$ and $Y$ are metric spaces.
Given a subset $W\subset X\times Y$ and an \usc\ nondecreasing function $\mu:[0,+\infty]\to[0,+\infty]$,
we use the statements derived in Section~\ref{bas_est} to characterize regularity of $F$ \emph{on $W$ with functional modulus $\mu$}.
We ``set free'' the remaining two variables $x$ and $y$ restricting them to the set $W$.

\begin{definition}\label{funreg}
\begin{enumerate}
\item[\rm(i)]
$F$ is regular on $W$ with functional modulus $\mu$ if
\begin{equation*}
d(x,F_0\iv(y))\le\mu(\delta(y,F,x))
\quad\mbox{for all}\quad
(x,y)\in W.
\end{equation*}
\item[\rm(ii)]
$F$ is
open on $W$ with functional modulus $\mu$ if
\begin{equation*}
y\in F(B(x,t),0)
\quad\mbox{for all}\quad
(x,y)\in W\mbox{ and } t>\mu(\delta(y,F,x)).
\end{equation*}
\end{enumerate}
\end{definition}

The above properties differ from the conventional metric regularity defined for set-valued mappings between metric spaces (cf. \cite{DonRoc14}) and its nonlinear extensions (cf. \cite{Iof13}).
The relationship between the two settings will be discussed in Section~\ref{CON}.

The next proposition is a consequence of Proposition~\ref{P2.1} thanks to Remark~\ref{R2.5}.

\begin{proposition}\label{reg_rel}
The two properties in Definition~\ref{funreg} are equivalent.
\end{proposition}

\begin{remark}
It follows from Proposition~\ref{P2.1} that the properties in Definition~\ref{funreg} are implied by the following stronger version of openness:
\begin{equation*}
y\in F(B(x,\mu(\delta(y,F,x))),0)
\quad\mbox{for all}\quad
(x,y)\in W.
\end{equation*}
\end{remark}

The criteria of regularity in the next theorem are direct consequences of Theorems~\ref{Kh} and \ref{Kh4} and Corollary~\ref{Kh4.2}.

\begin{theorem}\label{T3.3}
Suppose that, for any $(x,y)\in W$, the mapping $\tau\mapsto F_\tau\iv(y)$ on $\R_+$ is outer semicontinuous at $0$ and,
for some $\ga>\delta(y,F,x)$ and any $t\in(0,\ga)$
with $(x,t,y)\in \gr F$, one of the following sets of conditions is satisfied:
\begin{enumerate}
\item[\rm(i)]
there are sequences of positive numbers $(b_n)$ and $(c_n)$ and a function $m:(0,\infty)\to (0,\infty)$ such that conditions
\eqref{B3}--\eqref{B5+} and \eqref{A4-} hold true,
\item[\rm(ii)]
there are functions $b,m:(0,\infty)\to (0,\infty)$ such that condition \eqref{mutau+} is satisfied and, for any $\tau>0$ with $\mu(\tau)\le \mu(t)$, conditions \eqref{mu++} and \eqref{net+++} hold true,
\item[\rm(iii)]
there are functions $b,m:(0,\infty)\to (0,\infty)$ such that condition \eqref{mutau+} is satisfied and, for any $\tau>0$ with $\mu(\tau)\le \mu(t)$, conditions \eqref{mu++}, \eqref{set1} and \eqref{set2} hold true.
\end{enumerate}
Then, $F$ is regular on $W$ with functional modulus $\mu$.
\end{theorem}

In the next statement, which is a consequence of the ``parameter-free'' Theorem~\ref{T2.6}, $p_Y:X\times Y\rightarrow Y$ denotes the canonical projection on $Y$:
for any $(x,y)\in X\times Y$, $p_Y(x,y)=y$.
Given a pair $(x,y)\in W$, denote
$$
U_{x,y}:=\{u\in X\mid \delta(y,F,u)>0,\; \mu(\delta(y,F,u))+d(u,x)\le\mu(\delta(y,F,x))\}.
$$

\begin{theorem}\label{T3.5}
Let $\mu$ be continuous,
$\mu(\tau)=0$ if and only if $\tau=0$.
Suppose that $F^{-1}$ is closed-valued on $p_Y(W)$ and, for any $(x,y)\in W$ and
$u\in U_{x,y}$,
there exists a point $u'\neq u$ such that
\begin{equation}\label{4''}
 \mu(\delta(y,F,u'))\le \mu(\delta(y,F,u))-d(u,u').
\end{equation}
Then, $F$ is regular on $W$ with functional modulus $\mu$.
\end{theorem}
\begin{proof}
Fix an arbitrary $(x,y)\in W$.
%
%
We need to show that $d(x,F_0\iv(y))\le\mu(\delta(y,F,x))$.
If there exists a point $u$ such that $\delta(y,F,u)=0$ and $d(x,u)\le \mu(\delta(y,F,x))$ (in particular, if $\delta(y,F,x)=0$), then, by the closedness of $F^{-1}(y)$, we have $u\in F_0^{-1}(y)$, and the inequality holds trivially.

Suppose that $\delta(y,F,u)>0$ for any $u\in X$ such that $d(x,u)\le \mu(\delta(y,F,x))$.
Take any $u\in X$ such that $d(x,u)\le \mu(\delta(y,F,x))-\mu(\delta(y,F,u))$ and any $\tau\in(0,\delta(y,F,x)]$ such that $(u,\tau)\in F\iv(y)$.
Then,
$\tau\ge \delta(y,F,u)>0$ and,
by the assumption, there exists a point $u'\neq u$ satisfying \eqref{4''}.
Setting $\tau'=\delta(y,F,u')$, we get $(u',\tau')\in F^{-1}(y)$ and condition  \eqref{mu-} is satisfied:
$$
\mu(\tau')=\mu(\delta(y,F,u'))\le \mu(\delta(y,F,u))-d(u,u')\le \mu(\tau)-d(u,u').
$$
The mapping $\tau\mapsto F_\tau\iv(y)$ is outer semicontinuous on $[0,\delta(y,F,x)]$ thanks to the closedness of $F^{-1}(y)$.
The required inequality follows from Theorem \ref{T2.6}.
\end{proof}

One can define seemingly more general $\nu$-versions of the properties in Definition~\ref{funreg}, determined by a function $\nu:W\to (0,\infty]$; see \cite{Iof13} for the motivations behind such properties.

\begin{definition}\label{funreg+}
\begin{enumerate}
\item[\rm(i)]
$F$ is $\nu$-regular on $W$ with functional modulus $\mu$ if
\begin{equation*}
\hspace{-.5cm}
d(x,F_0\iv(y))\le\mu(\delta(y,F,x))
\;\;\mbox{for all}\;
(x,y)\in W
\mbox{ with } \mu(\delta(y,F,x))<\nu(x,y).
\end{equation*}
\item[\rm(ii)]
$F$ is $\nu$-open on $W$ with functional modulus $\mu$ if
\begin{equation*}
y\in F(B(x,t),0)
\quad\mbox{for all}\quad
(x,y)\in W\mbox{ and } t\in(\mu(\delta(y,F,x)),\nu(x,y)).
\end{equation*}
\end{enumerate}
\end{definition}

\begin{remark}\label{gam}
Each of the properties in Definition~\ref{funreg}
is a particular case of the corresponding one in Definition~\ref{funreg+}
with any function $\nu:W\to (0,\infty]$ satisfying $\mu(\delta(y,F,x))<\nu(x,y)$
for all $(x,y)\in W$ with $\mu(\delta(y,F,x))<+\infty$, e.g., one can take $\nu\equiv+\infty$.
At the same time, each of the properties in Definition~\ref{funreg+}
can be considered as a particular case of the corresponding one in Definition~\ref{funreg}
with the set $W$ replaced by $W':=\{(x,y)\in W\mid \mu(\delta(y,F,x))<\nu(x,y)\}$.
\end{remark}


\begin{proposition}\label{reg_rel.1}
The two properties in Definition~\ref{funreg+} are equivalent.
\end{proposition}

We next
formulate the corresponding criteria for $\nu$-regularity.
The next two theorems are consequences of Theorem~\ref{T3.3} and the ``parameter-free'' Theorem~\ref{T3.5}, respectively, thanks to Remark~\ref{gam} and the simple observation that, if $\mu(\delta(y,F,x))<\nu(x,y)$, then, making use of the upper semicontinuity of $\mu$, it is possible to choose a $\gamma>\delta(y,F,x)$ such that $\mu(\gamma)<\nu(x,y)$.

\begin{theorem}\label{T3.3'}
Suppose that, for any $(x,y)\in W$, the mapping $\tau\mapsto F_\tau\iv(y)$ on $\R_+$ is outer semicontinuous at $0$ and,
for any $t>0$ with $(x,t,y)\in \gr F$ and $\mu(t)<\nu(x,y)$, one of the three sets of conditions in Theorem~\ref{T3.3} is satisfied.
Then, $F$ is $\nu$-regular on $W$ with functional modulus $\mu$.
\end{theorem}

\begin{theorem}\label{T3.5"}
Let $\mu$ be continuous, $\mu(\tau)=0$ if and only if $\tau=0$ and $\nu:
\bigcup_{(x,y)\in W}(U_{x,y} \times \{y\})
\to (0,\infty)$ be Lipschitz
continuous with
modulus not greater than 1 in
$x$ for any $y\in p_Y(W)$.
Suppose that $F^{-1}$ takes closed values on $p_Y(W)$ and, for any
$(x,y)\in W$ and
$u\in U_{x,y}$
with
$
\mu(\delta(y,F,u))<\nu(u,y)$, there exists a point $u'\neq u$ such that condition \eqref{4''} holds true.
Then, $F$ is $\nu$-re\-gular on $W$ with functional modulus $\mu$.
\end{theorem}
\begin{proof}
Define $W':=\{(x,y)\in W\mid \mu(\delta(y,F,x))<\nu(x,y)\}$
and take any $(x,y)\in W'$ and
$u\in U_{x,y}$.
Then,
taking into account the Lipschitz continuity of $\nu$, we have:
$$
\mu(\delta(y,F,u))\le \mu(\delta(y,F,x))-d(x,u)<\nu(x,y)-d(x,u)\le \nu(u,y).
$$
Hence, there exists a point $u'\neq u$ such that \eqref{4''} holds true.
By Theorem~\ref{T3.5}, $F$ is regular on $W'$ and,
thanks to Remark \ref{gam},
$\nu$-re\-gular on $W$ with functional modulus $\mu$.
\end{proof}

\begin{remark}
The properties in Definitions~\ref{funreg} and \ref{funreg+} depend on the choice of the set $W$ and (in the case of Definitions~\ref{funreg+}) function $\nu$.
Changing these parameters may lead to the change of the regularity modulus or even kill regularity at all; cf. \cite[Example~1]{Iof13}.
\end{remark}

The next definition introduces the more conventional local versions of the properties in Definition~\ref{funreg} related to a fixed point $(\bx,\by)\in\gph F_0$.

\begin{definition}\label{3.3}
\begin{enumerate}
\item[\rm(i)]
$F$ is regular at $(\bx,\by)$ with functional modulus $\mu$ if there exist neighbourhoods $U$ of $\bx$ and $V$ of $\by$ such that
\begin{equation*}
d(x,F_0\iv(y))\le\mu(\delta(y,F,x))
\quad\mbox{for all}\quad
x\in U,\;y\in V.
\end{equation*}
\item[\rm(ii)]
$F$ is open at $(\bx,\by)$ with functional modulus $\mu$ if there exist neighbourhoods $U$ of $\bx$ and $V$ of $\by$ such that
\begin{equation*}
y\in F(B(x,t),0)
\quad\mbox{for all}\quad
x\in U,\;y\in V\mbox{ and } t>\mu(\delta(y,F,x)).
\end{equation*}
\end{enumerate}
\end{definition}

The properties in Definition~\ref{3.3} are obviously equivalent to the corresponding ones in Definition~\ref{funreg} with $W:=U\times V$.
The next three statements are consequences of Proposition~\ref{reg_rel} and Theorems~\ref{T3.3} and \ref{T3.5}, respectively.

\begin{proposition}\label{T3.9}
The two properties in Definition~\ref{3.3} are equivalent.
\end{proposition}

\begin{theorem}\label{T3.10}
Suppose that there exist neighbourhoods $U$ of $\bx$ and $V$ of $\by$ such that, for any $x\in U$ and $y\in V$, the mapping $\tau\mapsto F_\tau\iv(y)$ on $\R_+$ is outer semicontinuous at $0$ and,
for some $\ga>\delta(y,F,x)$ and any $t\in(0,\ga)$
with $(x,t,y)\in \gr F$, one of the three sets of conditions in Theorem~\ref{T3.3} is satisfied.
Then, $F$ is regular at $(\bx,\by)$ with functional modulus $\mu$.
\end{theorem}

\begin{theorem}\label{T3.12}
Let $\mu$ be continuous,
$\mu(\tau)=0$ if and only if $\tau=0$.
Suppose that there exist neighbourhoods $U$ of $\bx$ and $V$ of $\by$ such that $F^{-1}$ takes closed values on $V$ and, for any $x\in U$, $y\in V$, and
$u\in U_{x,y}$,
there exists a point $u'\neq u$ such that condition
\eqref{4''} is satisfied.
Then, $F$ is regular at $(\bx,\by)$ with functional modulus $\mu$.
\end{theorem}

\section{Conventional setting}\label{CON}

In this section, we consider the standard in variational analysis setting of a set-valued mapping $F:X\rightrightarrows Y$ between metric spaces.
Such a mapping can be imbedded into the more general setting explored in the previous sections by defining a set-valued mapping $\mathcal{F}:X\times \R_+\rightrightarrows Y$ as follows (cf. \cite[p.~508]{Iof00_}: for any $x\in X$ and $t\ge0$,
\begin{equation}\label{F}
\mathcal{F}(x,t):=B(F(x),t)=
\begin{cases}
\{y\in Y\mid d(y,F(x))<t\}
&\mbox{ if } t>0,\\
F(x) &\mbox{ if } t=0.
\end{cases}
\end{equation}
(Recall the convention: $B(y,0)=\{y\}$.)
We are going to consider also mappings $\overline{F}:X\rightrightarrows Y$ and $\overline{\mathcal{F}}:X\times \R_+\rightrightarrows Y$, whose values are the closures of the corresponding values of $F$ and $\mathcal{F}$, respectively:
$\overline{F}(x):=\overline{F(x)}$ and
\begin{equation}\label{clF}
\overline{\mathcal{F}}(x,t):=\overline{B}(F(x),t)=
\begin{cases}
\{y\in Y\mid d(y,F(x))\le t\}
&\mbox{ if } t>0,\\
\overline{F(x)} &\mbox{ if } t=0.
\end{cases}
\end{equation}

The next proposition summarizes several simple facts with regard to the relationship between $F$, $\mathcal{F}$ and $\overline{\mathcal{F}}$.

\begin{proposition}\label{R6.1}
\begin{enumerate}
\item[\rm(i)]
$\mathcal{F}_0(x)={F(x)}$, $\overline{\mathcal{F}}_0(x)=\overline{F(x)}$ for all $x\in X$.
\item[\rm(ii)]
$\delta(y,\mathcal{F},x)= \delta(y,\overline{\mathcal{F}},x)=d(y,F(x))$ for all $x\in X$ and $y\in Y$.
\item[\rm(iii)]
$ \mathcal{F}_0^{-1}(B(y,t))=F\iv(B(y,t))= \mathcal{F}\iv_t(y)$ for all $y\in Y$ and $t\ge0$.
\item[\rm(iv)]
$\overline{\mathcal{F}}{}_0\iv(\overline{B}(y,t))= \overline{F}{}\iv(\overline{B}(y,t))\subset \overline{\mathcal{F}}{}\iv_t(y)$ for all $y\in Y$ and $t\ge0$.
\item[\rm(v)]
If $F^{-1}$ is closed at $y$, then the mappings $\tau\mapsto\mathcal{F}_\tau\iv(y)$ and $\tau\mapsto\overline{\mathcal{F}}{}_\tau\iv(y)$ on $\R_+$ are outer semicontinuous at $0$.
\item[\rm(vi)]
For any $y\in Y$ and $\tau>0$, $\mathcal{F}$ and $\overline{\mathcal{F}}$ satisfy condition \eqref{set1}.
\item[\rm(vii)]
If $F$ is upper semicontinuous on $X$, i.e., for any $x\in X$ and $\eps>0$, there exists a $\de>0$ such that $F(u)\subset B(F(x),\eps)$ for all $u\in B(x,\de)$,
then $\overline{\mathcal{F}}{}\iv$ is closed-valued.
In particular, for any $y\in Y$, the mapping $\tau\mapsto \overline{\mathcal{F}}{}_\tau\iv(y)$ is outer semicontinuous on $\R_+$.
\end{enumerate}
\end{proposition}

\begin{proof}
(i) The equalities make part of definitions~\eqref{F} and \eqref{clF}.

(ii) By \eqref{del}, \eqref{F} and \eqref{clF},
\begin{gather*}
\delta(y,\mathcal{F},x)=\inf\{t>0\mid d(y,F(x))< t\}=d(y,F(x)),
\\
\delta(y,\overline{\mathcal{F}},x)=\inf\{t>0\mid d(y,F(x))\le t\}=d(y,F(x)).
\end{gather*}

(iii) If $t=0$, then $\mathcal{F}\iv_0(y)={F}\iv(y)$ and both equalities hold true automatically for all $y\in Y$.
If $t>0$, then
$$
x\in\mathcal{F}\iv_t(y) \;\Leftrightarrow\;
d(y,F(x))<t
\;\Leftrightarrow\;
{F(x)}\cap B(y,t)\ne\emptyset
\;\Leftrightarrow\;
x\in F\iv(B(y,t)).
$$
Hence, $\mathcal{F}\iv_t(y)=F\iv(B(y,t))$.
The other equality is satisfied because $\mathcal{F}\iv_0(v)={F}\iv(v)$ for all $v\in B(y,t)$.

(iv) If $t=0$, then $\overline{\mathcal{F}}{}\iv_0(y)= \overline{\mathcal{F}}{}\iv_0(\overline{B}(y,0))= \overline{F}{}\iv(y)$ for all $y\in Y$.
If $t>0$, then
$$
x\in\overline{F}{}\iv(\overline{B}(y,t))
\;\Leftrightarrow\;
{\overline{F}(x)}\cap\overline{B}(y,t)\ne\emptyset
\;\Rightarrow\;
d(y,F(x))\le t
\;\Leftrightarrow\;
x\in\overline{\mathcal{F}}{}\iv_t(y).
$$
Hence, $\overline{F}{}\iv(\overline{B}(y,t))\subset \overline{\mathcal{F}}{}\iv_t(y)$.
The claimed equality is satisfied because $\overline{\mathcal{F}}{}\iv_0(v)= \overline{F}{}\iv(v)$ for all $v\in\overline{B}(y,t)$.

(v) If $x_n\to z$ and $t_n\downarrow 0$ with $d(y,F(x_n))<t_n$ $(n=1,2,\ldots)$, then, for any $n$, there exists a $y_n\in F(x_n)$ such that $d(y,y_n)<t_n$.
Hence, $y_n\to y$ as $n\to \infty$.
Since $F^{-1}$ is closed at $y$, we have $z\in F^{-1}(y)$ and consequently $y\in F(z)= \mathcal{F}(z,0)$.

Similarly, if $x_n\to z$ and $t_n\downarrow 0$ with $d(y,F(x_n))\le t_n$ $(n=1,2,\ldots)$, then, for any $n$, there exists a $y_n\in F(x_n)$ such that $d(y,y_n)<2t_n$.
Hence, $y_n\to y$ as $n\to \infty$.
Since $F^{-1}$ is closed at $y$, we have $z\in F^{-1}(y)$ and consequently $y\in F(z)\subset \mathcal{F}(z,0)$.

(vi) follows from (iii) and (iv).

(vii) If $y\in Y$, $x_n\to z$ and $t_n\to\tau$ with $d(y,F(x_n))\le t_n$ $(n=1,2,\ldots)$, then, since $F$ is upper semicontinuous,
$$
d(y,F(z))\le \liminf_{n\to\infty}d(y,F(x_n))\le \lim_{n\to\infty}t_n=\tau,
$$
that is, $y\in\overline{\mathcal{F}}(z,\tau)$.
\end{proof}

Thanks to parts (i) and (ii) of Proposition~\ref{R6.1}, the definitions of regularity and openness properties explored in the previous sections in the current setting can be expressed in metric terms.
In the next definition, which corresponds to a group of definitions from Section~\ref{MET_SEC}, $\mu:[0,+\infty]\to[0,+\infty]$ is an \usc\ nondecreasing function playing the role of a \emph{modulus} of the corresponding property.

\begin{definition}\label{D6.1}
\begin{enumerate}
\item[\rm(i)]
Given a set $W\subset X\times Y$, mapping
$F$ is metrically regular on $W$ with functional modulus $\mu$ if
\begin{equation}\label{D6.1.1}
d(x,F\iv(y))\le\mu(d(y,F(x)))
\quad\mbox{for all}\quad
(x,y)\in W.
\end{equation}
\item[\rm(ii)]
Given a set $W\subset X\times Y$, mapping
$F$ is
open on $W$ with functional modulus $\mu$ if
\begin{equation*}
y\in F(B(x,t))
\quad\mbox{for all}\quad
(x,y)\in W\mbox{ and } t>\mu(d(y,F(x))).
\end{equation*}
\item[\rm(iii)]
Given a set $W\subset X\times Y$ and a function $\nu:W\to (0,\infty]$, mapping
$F$ is metrically $\nu$-regular on $W$ with functional modulus $\mu$ if
\begin{align}\label{D6.1.3}
d(x,F\iv(y))\le\mu(d(y,F(x)))
\;\;
&\mbox{for all}\;
(x,y)\in W
\\&\notag
\mbox{with } \mu(d(y,F(x)))<\nu(x,y).
\end{align}
\item[\rm(iv)]
Given a set $W\subset X\times Y$ and a function $\nu:W\to (0,\infty]$, mapping
$F$ is $\nu$-open on $W$ with functional modulus $\mu$ if
\begin{equation*}
y\in F(B(x,t))
\quad\mbox{for all}\quad
(x,y)\in W\mbox{ and } t\in(\mu(d(y,F(x))),\nu(x,y)).
\end{equation*}
\item[\rm(v)]
$F$ is metrically regular at a point $(\bx,\by)\in\gph F$ with functional modulus $\mu$ if there exist neighbourhoods $U$ of $\bx$ and $V$ of $\by$ such that
\begin{equation}\label{D6.1.5}
d(x,F\iv(y))\le\mu(d(y,F(x)))
\quad\mbox{for all}\quad
x\in U,\;y\in V.
\end{equation}
\item[\rm(vi)]
$F$ is
open at $(\bx,\by)\in\gph F$ with functional modulus $\mu$ if there exist neighbourhoods $U$ of $\bx$ and $V$ of $\by$ such that
\begin{equation}\label{D6.1.6}
y\in F(B(x,t))
\quad\mbox{for all}\quad
x\in U,\;y\in V\mbox{ and } t>\mu(d(y,F(x))).
\end{equation}
\end{enumerate}
\end{definition}

\begin{remark}
If $\mu$ is strictly increasing, then condition \eqref{D6.1.6} can be rewritten equivalently in a more conventional ``openness-like'' form (cf. \cite{Iof13}):
\begin{equation*}
B(F(x),\mu\iv(t))\cap V\subset F(B(x,t))
\quad\mbox{for all}\quad
x\in U\mbox{ and } t>0.
\end{equation*}
In the case $W=U\times V$, similar simplifications can be made also in parts (ii) and (iv) of the above definition.
\end{remark}

In the linear case ($\mu$ is a linear function), the metric regularity and openness/cove\-ring properties in the above definition are well known in both local and global settings (cf., e.g., \cite{RocWet98,Iof00_,Mor06.1,DonRoc14}) including regularity on a set \cite{Iof00_,Iof10}.
The nonlinear setting in the above definition follows Ioffe \cite{Iof13} where the properties in parts (iii) and (iv), were mostly investigated in the particular case  $W=U\times V$ where $U\subset X$ and $V\subset Y$ and the function $\nu$ depends only on $x$.

Observe that condition \eqref{D6.1.1} in Definition~\ref{D6.1} is equivalent to
\begin{equation*}
d(x,F\iv(y))\le\mu(d(y,y'))\;\;
\mbox{for all}\;
(x,y)\in W
\;\mbox{and}\;
y'\in F(x).
\end{equation*}
In its turn, condition $y'\in F(x)$ is equivalent to $x\in F\iv(y')$.
This and similar observations regarding conditions \eqref{D6.1.3} and \eqref{D6.1.5} allow us to rewrite these conditions, respectively, as follows:
\begin{align*}
d(x,F\iv(y_2))\le\mu(d(y_1,y_2))\;\;
&\mbox{for all}\;
y_1,y_2\in Y,\;x\in F\iv(y_1)
\;\mbox{with}\;
(x,y_2)\in W,
\\
d(x,F\iv(y_2))\le\mu(d(y_1,y_2))\;\;
&\mbox{for all}\;
y_1,y_2\in Y,\;x\in F\iv(y_1)
\\&
\mbox{with}\;
(x,y_2)\in W,\;\mu(d(y_1,y_2))<\nu(x,y_2),
\\
d(x,F\iv(y_2))\le\mu(d(y_1,y_2))\;\;
&\mbox{for all}\;
y_1\in Y,\;y_2\in V,\;x\in F\iv(y_1)\cap U.
\end{align*}

Thanks to these observations, one can complement the regularity and openness properties in Definition~\ref{D6.1} with the corresponding H\"older-like (Aubin in the linear case) properties.

In the definition below,
$\mu:[0,+\infty]\to[0,+\infty]$ is again an \usc\ nondecreasing function.

\begin{definition}\label{D6.4}
\begin{enumerate}
\item[\rm(i)]
Given a set $W\subset X\times Y$, mapping
$F$ is H\"older on $W$ with functional modulus $\mu$ if
\begin{align*}
d(y,F(x_2))\le\mu(d(x_1,x_2))\;\;
&\mbox{for all}\;
x_1,x_2\in X,\;y\in F(x_1)
\;\mbox{with}\;
(x_2,y)\in W.
\end{align*}
\item[\rm(ii)]
Given a set $W\subset X\times Y$ and a function $\nu:W\to (0,\infty]$, mapping
$F$ is $\nu$-H\"older on $W$ with functional modulus $\mu$ if
\begin{align*}
d(y,F(x_2))\le\mu(d(x_1,x_2))\;\;
&\mbox{for all}\;
x_1,x_2\in X,\;y\in F(x_1)
\\&
\mbox{with}\;
(x_2,y)\in W,\; \mu(d(x_1,x_2))<\nu(x_2,y).
\end{align*}
\item[\rm(iii)]
$F$ is H\"older at a point $(\bx,\by)\in\gph F$ with functional modulus $\mu$ if there exist neighbourhoods $U$ of $\bx$ and $V$ of $\by$ such that
\begin{align}\label{D6.4.3}
d(y,F(x_2))\le\mu(d(x_1,x_2))\;\;
&\mbox{for all}\;
x_1,x_2\in U,\;y\in F(x_1)\cap V.
\end{align}
\end{enumerate}
\end{definition}

Thanks to
Propositions~\ref{reg_rel}, \ref{reg_rel.1}, \ref{T3.9} and the discussion before Definition~\ref{D6.4}, we have the following list of equivalences.

\begin{theorem}\label{T6.1}
Suppose $\mu:[0,+\infty]\to[0,+\infty]$ is an \usc\ increasing function.
\begin{enumerate}
\item[\rm(i)]
Given a set $W\subset X\times Y$, properties (i) and (ii) in Definition~\ref{D6.1} are equivalent to $F\iv$ being H\"older on
\begin{equation}\label{T6.1.1}
W':=\{(y,x)\in Y\times X\mid (x,y)\in W\}
\end{equation}
with functional modulus $\mu$.
\item[\rm(ii)]
Given a set $W\subset X\times Y$, properties (iii) and (iv) in Definition~\ref{D6.1} are equivalent to $F\iv$ being $\nu'$-H\"older on \eqref{T6.1.1} with functional modulus $\mu$, where $\nu':W'\to (0,\infty]$ is defined by the equality $\nu'(y,x)=\nu(x,y)$.
\item[\rm(iii)]
Given a point $(\bx,\by)\in\gph F$,
properties (v) and (vi) in Definition~\ref{D6.1} are equivalent to $F\iv$ being H\"older at $(\by,\bx)$
with functional modulus $\mu$.
\end{enumerate}
\end{theorem}

\begin{remark}
Most of the equivalences in Theorem~\ref{T6.1} hold true with function $\mu$ nondecreasing.
The assumption that $\mu$ is strictly increasing is only needed in part (iii).
In fact, it follows from the discussion before Definition~\ref{D6.4}, that properties (v) and (vi) in Definition~\ref{D6.1} are equivalent to a stronger version of the H\"older property of $F\iv$ which corresponds to replacing condition \eqref{D6.4.3} in Definition~\ref{D6.4} by the following one:
\begin{align*}
d(y,F(x_2))\le\mu(d(x_1,x_2))\;\;
&\mbox{for all}\;
x_1\in X,\;x_2\in U,\;y\in F(x_1)\cap V.
\end{align*}
If $\mu$ is strictly increasing, then the two versions are equivalent.
\end{remark}

We next formulate several regularity criteria in the conventional setting of a mapping $F:X\rightrightarrows Y$ between metric spaces.
All of them are consequences of the corresponding statements in Section~\ref{MET_SEC} thanks to the relationships in Proposition~\ref{R6.1}.
From now on, we assume that $X$ is complete.

\begin{theorem}\label{T6.2}
Given a set $W\subset X\times Y$, suppose that, for any $(x,y)\in W$, the inverse mapping $F^{-1}$ is closed at $y$ and,
for some $\ga>d(y,F(x))$ and any $t\in(0,\ga)$,
 one of the following sets of conditions is satisfied:
\begin{enumerate}
\item[\rm(i)]
there are sequences of positive numbers $(b_n)$ and $(c_n)$ and a function $m:(0,\infty)\to (0,\infty)$ such that conditions
\eqref{B3} and \eqref{A4-} hold true and
\begin{align*}
d\left(x,F^{-1}(B(y,m(c_{1})))\right)<&b_0,
\\
d\left(u,F^{-1}(B(y,m(c_{n+1})))\right)<&b_n\\
\mbox{for all}\;\; u\in F^{-1}(B&(y,m(c_n)))\cap B(x,\sum_{i=0}^{n-1}b_{i})\; (n=1,2,\ldots),
\end{align*}
\item[\rm(ii)]
there are functions $b,m:(0,\infty)\to (0,\infty)$ such that condition \eqref{mutau+} is satisfied and, for any $\tau>0$ with $\mu(\tau)\le \mu(t)$, condition \eqref{mu++} holds true and
\begin{align*}
d\left(u,F^{-1}(B(y,b(\tau)))\right)<m(\tau) \mbox{ for all } u\in F^{-1}(B(y,\tau))\cap B(x,\mu(t)-\mu(\tau)),
\end{align*}
\item[\rm(iii)]
there are functions $b,m:(0,\infty)\to (0,\infty)$ such that condition \eqref{mutau+} is satisfied and, for any $\tau>0$ with $\mu(\tau)\le \mu(t)$, condition \eqref{mu++} holds true and
$$
d\left(y,F(B(u,m(\tau)))\right)<b(\tau) \mbox{ for all } u\in F^{-1}(B(y,\tau))\cap B(x,\mu(t)-\mu(\tau)).
$$
\end{enumerate}
Then, $F$ is metrically regular on $W$ with functional modulus $\mu$.
\end{theorem}

\begin{theorem}\label{T6.4}
Let $\mu$ be continuous,
$\mu(\tau)=0$ if and only if $\tau=0$.
Given a set $W\subset X\times Y$, suppose that $F$ is \usc\ and, for any $(x,y)\in W$ and
$u\in X$ such that $d(y,F(u))>0$ and $\mu(d(y,F(u)))+d(u,x)\le\mu(d(y,F(x)))$,
there exists a point $u'\neq u$ such that
\begin{equation*}
\mu(d(y,F(u')))\le \mu(d(y,F(u)))-d(u,u').
\end{equation*}
Then, $F$ is metrically regular on $W$ with functional modulus $\mu$.
\end{theorem}
\begin{proof}
By Theorem \ref{T3.5} and Proposition~\ref{R6.1}(i), (ii) and (vii), set-valued mapping $\mathcal{\overline{F}}$ is regular on $W$ with functional modulus $\mu$.
Since $F$ is \usc, it is closed-valued and consequently making use of Proposition~\ref{R6.1}(i) again, we have for any $y\in Y$ that $\mathcal{\overline{F}}{}^{-1}_0(y)= \overline{F}{}^{-1}(y)= F^{-1}(y)$.
Hence, the regularity of $\mathcal{\overline{F}}$ is equivalent to the metric regularity of $F$.
\end{proof}

\section{Optimality conditions}\label{S5}

In this section, we apply our general nonlinear regularity model to establish second-order necessary optimality conditions for a nonsmooth set-valued optimization problem with mixed constraints.

Let $X,Y,Z$ and $W$ be Banach spaces; $S$ a nonempty subset of $X$; $C$ a proper convex ordering cone in $Y$ expressing the objective preference in the set-valued optimization problem below (``proper" means $C\ne\emptyset$ and $C\ne Y$); $D$ a convex cone with nonempty interior in $Z$; $F: X\rightrightarrows Y$, $G:X\rightrightarrows Z$, and $H: X\rightrightarrows W$ set-valued mappings.
We consider the problem
$$
\mbox{Minimize}_C\;F(x)\quad\mbox{subject to}\quad x\in\Omega,
$$
where
$$
\Omega:=\{x\in X\mid x\in S,\; G(x)\cap (-D)\neq \emptyset,\; 0\in H(x)\}.
$$
A triple $(\bx,\by,\bz)$ is said to be
\emph{feasible} if $\bx\in \Omega$, $\by\in F(x)$ and $\bz\in G(x)\cap (-D)$.
Alongside the ordering cone $C$ we consider another proper open cone $Q\subset Y$.
A point $(\bar x,\bar y)\in X\times Y$ is called a local \emph{$Q$-solution} if
\begin{align}\label{5.3}
F(U\cap\Omega)\cap(\bar y-Q)=\emptyset
\end{align}
for some neighbourhood $U$ of $\bar x$.

The above problem subsumes various vector- and set-valued optimization problems while the concept of $Q$-solution, under a suitable choice of $Q$, subsumes various kinds of solutions; cf. \cite{KhaTun15}.
For instance, if $Q=\Int C\ne\emptyset$, then $Q$-solution coincides with the conventional (local) \emph{weak} solution.
If $Q$ is an open cone such that $C\setminus\{0\}\subset Q$, then $Q$-solution becomes \emph{Henig proper} solution.
Similarly, setting $Q=Y \setminus (-\overline{\cone}(F(U\cap\Omega)-\bar y+C))$ where $U$ is a neighbourhood of $\bar x$, we come to the concept of \emph{Benson proper} solution.

It is worth noting the two specific features of the second-order necessary condition we present below: the regularity condition plays an important role and the right-hand side of the multiplier rule \eqref{5.1} is not the number $0$ as in the classical result (and also in many its developments until now) and it may be strictly negative in particular cases.
This phenomenon, known as the \emph{envelope-like effect} revealed by Kawasaki \cite{Kaw88}, may happen because of the presence of the closure sign in the definition of the set of critical directions \eqref{critical}.
For typical contributions to optimality conditions with these two features, we refer the reader to the references \cite{Com90,AruAvaIzm08,BonComSha99,GutJimNov10,Jou94, Kaw88,KhaTua13.2,KhaTua13.3,Pen99}.
Theorem~\ref{T5.2} below is a further development of many results in these references.

We first recall several useful definitions.
\begin{enumerate}
\item [{\rm(i)}]
The (positive) \emph{dual cone} to a cone $K$ in $X$:
$$
K^*:=\{x^* \in X^*\mid\langle x^*,x\rangle\geq 0,\; \forall x\in K\}.
$$
\item [{\rm(ii)}]
The \emph{contingent}, \emph{interior tangent} and \emph{normal} cones to a nonempty subset $M\subset X$ at $\bx\in\overline M$:
\begin{align*}
T(M,\bx):=&\{u\in X\mid\exists \gamma _n \downarrow0,\; u_n\to u
\mbox{ such that }
\bx+\gamma_nu_n \in M,\;\forall n\},
\\
IT(M,\bx):=&\{u\in X\mid\forall \gamma_n\downarrow0,\; u_n\to u,
\mbox{ it holds }
\bx+\gamma_nu_n \in M,\;\forall
\mbox{ large }n\},
\\
N(M,\bx):=&-[T(M,\bx)]^*.
\end{align*}
\item [{\rm(iii)}]
The \emph{second-order contingent, adjacent}
and \emph{interior} sets to a nonempty subset $M\subset X$ at $\bx\in\overline M$ in a direction $u\in X$:
\begin{align*}
\hspace{-.5cm}
T^{2}(M,\bx,u) :=&\{x \in
X\mid \exists \gamma_{n} \downarrow 0,\; x_{n}
\rightarrow x,
\mbox{ s.t. }
\bx + \gamma_{n}u +\dfrac{1}{2}\gamma_{n}^{2}x_{n} \in M,\;\forall n\},
\\
\hspace{-.5cm}
A^{2}(M, \bx, u): =&\{x \in
X\mid \forall \gamma_{n} \downarrow 0,\; \exists x_{n}
\rightarrow x,
\mbox{ s.t. }
\bx + \gamma_{n}u +
\dfrac{1}{2}\gamma_{n}^{2}x_{n} \in M,\forall n\} ,
\\
\hspace{-.5cm}
IT^{2}(M, \bx, u) :=&\{x \in
X\mid \forall \gamma_{n} \downarrow 0,\; x_{n}
\rightarrow x,
\mbox{ it holds }
\bx + \gamma_{n}u +
\dfrac{1}{2}\gamma_{n}^{2}x_{n} \in M,
\\&\hspace{7.5cm}\forall
\mbox{ large }n\}.
\end{align*}
\item [{\rm(iv)}]
The \emph{outer limit} and \emph{inner limit} of a set-valued mapping $E: X\rightrightarrows Y$ at $\bx\in X$:
\begin{align*}
\Limsup_{x\to\bx} E(x):=&\{y\in Y\mid \liminf_{x\to\bx}d(y, E(x))=0\},
\\
\Liminf_{x\to\bx} E(x):=&\{y\in Y\mid \lim_{x\to\bx}d(y, E(x))=0\}.
\end{align*}
\item [{\rm(v)}]
The \emph{contingent} and \emph{lower} derivatives of a set-valued mapping $E: X\rightrightarrows Y$ at $(\bx,\by)\in\gph E$:
\begin{align*}
D E(\bx,\by)(x):=&\Limsup_{\gamma\downarrow0,\,x'\to x}\gamma^{-1}[E(\bx+\gamma x')-\bar y],
\\
D_{l}E(\bx,\by)(x):=&\Liminf_{\gamma\downarrow0,\,x'\to x}\gamma^{-1}[E(\bx+\gamma x')-\bar y],\;x\in X.
\end{align*}
\item [{\rm(vi)}]
The \emph{second-order contingent} and \emph{lower} derivatives of a set-valued mapping $E: X\rightrightarrows Y$ at $(\bx,\by)\in\gph E$ in a direction $(u,v)\in X\times Y$:
\begin{align*}
D^{2}E(\bx,\by,u,v)(x):=&\Limsup_{\gamma\downarrow 0,\,x'\to x}2\gamma^{-2}[E(\bx+\gamma{u}+\frac{1}{2}\gamma^2 x')-\bar y-\gamma v],
\\
D^{2}_{l}E(\bx,\by,u,v)(x):=&\Liminf_{\gamma\downarrow 0,\,x'\to x}2\gamma^{-2}[E(\bx+\gamma{u}+\frac{1}{2}\gamma^2 x')-\bar y-\gamma v],\;x\in X.
\end{align*}
\end{enumerate}

Note that, if $M$ is a convex set with $\Int M\ne\emptyset$ and $u\in T(M,\bx),u)$, then
\begin{gather}\label{ys0}
T(M,\bx)=\overline{IT(M,\bx)},
\quad
A^{2}(M,\bx,u)=\overline{IT^{2}(M,\bx,u)},
\\\label{ys}
A^2(M,\bx,u)+T(T(M,\bx),u)\subset A^2(M,\bx,u).
\end{gather}
If $K$ is a convex cone and $\bx\in\overline K$, then $$N(K,\bx)=\{x^* \in-K^*\mid \langle x^*,\bx\rangle=0\}.$$

Now we return to our optimization problem.
Assume that $Q$ is an open convex cone and denote $F_{+}(x):=F(x)+\overline Q$ and $G_{+}(x):=G(x)+D$.
For a feasible triple $(\bx,\by,\bz)$, we introduce the
set of \emph{critical directions}:
\begin{multline}\label{critical}
\mathcal{C}(\bx,\by,\bz):=\{(u,v,k)\in X\times Y\times Z\mid
v\in D_{l}F_{+}(\bx,\by)(u) \cap (-\bd Q),\\
k \in
D_{l}G_{+}(\bx,\bz)(u)\cap (-\overline{\cone}(D+\bz)),\; 0\in DH(\bx,0)(u)\}.
\end{multline}
Given a triple $(u,v,k)\in \mathcal{C}(\bar x,\bar y,\bar z)$ and a point $x\in X$, we denote
\begin{align*}
\Delta_{(u,v,k)}(x):=\left(D_{l}^{2}F_{+}(\bar x,\bar y,u,v),D_{l}^{2}G_{+}(\bar x,\bar z,u,k),D^{2}H(\bar x,0,u,0)\right)(x).
\end{align*}
In what follows, we will consider an extension of the mapping $H$: a set-valued mapping $\mathcal{H}:X\times \R_+\rightrightarrows W$ with the properties $\mathcal{H}_0(\cdot):=\mathcal{H}(\cdot,0)=H(\cdot)$ and (cf. definition~\eqref{del})
$$\delta(0,\mathcal{H},x):=\inf\{t>0\mid 0\in\mathcal{H}(x,t)\}\le\theta d(0,H(x))$$
for some $\theta>0$ and all $x$ in a neighbourhood of $\bar x$.
We will need to assume a kind of regular behaviour of this extension.

\begin{definition}\label{D5.1}
$\mathcal{H}$ is regular at $(\bx,\bar w)$ with functional modulus $\mu$ with respect to $S$ if there exist neighbourhoods $U$ of $\bx$ and $V$ of $\bar w$ such that
\begin{equation}\label{D5.1-1}
d(x,\mathcal{H}_0\iv(w)\cap S)\le\mu(\delta(w,\mathcal{H},x))
\quad\mbox{for all}\quad
x\in U\cap S,\;w\in V.
\end{equation}
\end{definition}

Observe that this is exactly the regularity in the sense of Definition~\ref{3.3}(i) for the restriction of the mapping $\mathcal{H}$ on $S\times \R_+$.
Recall that $\mu:[0,+\infty]\to[0,+\infty]$ is assumed upper semicontinuous and nondecreasing.
In what follows, we will assume additionally that $\limsup_{t\downarrow0}\mu(t)/t<\infty$.

\begin{theorem}\label{T5.2}
Let $(\bar x,\bar y)$ be a local $Q$-solution, $\bar z\in G(\bar x)\cap (-D)$, and $\mathcal{H}$ be regular at $(\bar x,0)$ with functional modulus $\mu$ with respect to $S$.
Suppose that $(u,v,k)\in \mathcal{C}(\bar x,\bar y,\bar z)$ and
$\Delta_{(u,v,k)}(IT^{2}(S,\bar x,u))$ is a convex set with nonempty interior.
Then, there exist multipliers $(v^*,k^*,w^* )\in Q^*\times
N(-D,\bar z)\times W^* \setminus \{(0,0,0)\}$ such that $\langle v^* ,v \rangle=\langle k^* ,k \rangle =0$ and
\begin{multline}\label{5.1}
v^* \circ D_{l}^{2}F_{+}(\bar x,\bar y,u,v)(x) +k^* \circ D_{l}^{2}G_{+}(\bar x,\bar z,u,k)(x)+w^* \circ D^{2}H(\bar x,0,u,0)(x)
\\
\geq \sup_{d \in A^2(-D,\bar z,k)}\langle k^*, d\rangle
\quad\mbox{for all}\quad
x \in IT^{2}(S,\bar x,u).
\end{multline}
Moreover, $v^* \neq 0$ if the following second-order constraint qualification holds:
\begin{multline}\label{5.2}
\cone\left((D_{l}^{2}G_{+}(\bar x,\bar z,u,k)-A^2(-D,\bar z,k),D^{2}H(\bar x,0,u,0))(IT^{2}(S,\bar x,u))\right)
\\
+\cone(D+\bz)\times\{0\}=Z\times W.
\end{multline}
\end{theorem}

\begin{proof}
We split the proof into several claims.

\underline{Claim 1}.
\emph{$(\bar x,\bar y)$ satisfies the primal necessary condition:}
$$D_{l}^{2}F_{+}(\bar x,\bar y,u,v)(T^2(\Omega,\bar x,u))\cap (-\cone(Q+v))=\emptyset.$$

Indeed, by the definition of $Q$-solution, \eqref{5.3} holds true for some neighbourhood $U$ of $\bar x$.
Let $x \in T^{2}(\Omega,\bar x,u)$ and $y \in D_{l}^{2}F_{+}(\bar x,\bar y,u,v)(x)$.
Then, there are $\gamma_n\downarrow 0$, $x_n \to x$, and $y_n\to y$ such that $\bar x+\gamma_nu+\dfrac{1}{2}\gamma^2_n x_n \in U\cap\Omega$ for all $n \in \mathbb{N}$ and $\bar y+\gamma_nv+\dfrac{1}{2}\gamma^{2}_n y_n\in
F(\bar x+\gamma_nu+\dfrac{1}{2}\gamma^{2}_n x_n)+\overline Q$ for all sufficiently large $n$.
Thanks to \eqref{5.3}, we have $\gamma_nv+\dfrac{1}{2}\gamma^{2}_n y_n\notin-Q$, and consequently, $y\notin-\cone(Q+v)$.
\smallskip

\underline{Claim 2}.
\emph{The following lower estimate for $T^2(\Omega,\bx,u)$ holds true:}
\begin{multline*}
\{x \in IT^2(S,\bx,u)\mid
D_{l}^2G_{+}(\bx,\bz,u,k)(x) \cap IT^2(-D,\bz,k) \ne \emptyset,
\\
0\in D^2H(\bx,0,u,0)(x)\}
\subset T^2(\Omega,\bx,u).
\end{multline*}

Suppose $x\in IT^2(S,\bx,u)$, $z\in D_{l}^2G_{+}(\bx,\bz,u,k)(x)\cap IT^2(-D,\bz,k)$ and $0\in D^2H(\bx,0,u,0)(x)$.
As $0\in D^2H(\bx,0,u,0)(x)$, there are $\gamma_n\downarrow 0$, $x_n \to x$, and $w_n \to 0$ such that $\dfrac{1}{2}\gamma_n^{2}w_n \in H(\bx+\gamma_nu+\dfrac{1}{2}\gamma_n^2x_n)$ for all $n\in\N$.
As $x\in IT^2(S,\bx,u)$, it holds $\bx+\gamma_nu+\dfrac{1}{2}\gamma_n^2x_n \in S$ for sufficiently large $n$.
As $\mathcal{H}$ is regular at $(\bx,0)$ with respect to $S$, for large $n$, we have:
\begin{align*}
d(\bx+\gamma_nu+\dfrac{1}{2}\gamma_n^2x_n, \mathcal{H}^{-1}_{0}(0)\cap S) &\leq \mu(\delta(0,
\mathcal{H},\bx+\gamma_nu+\dfrac{1}{2}\gamma_n^2x_n))
\\&
\leq \mu(\theta d(0,H(\bx+\gamma_{n}u+\dfrac{1}{2}\gamma_n^{2}x_n)))
\leq\mu\left(\dfrac{\theta}{2}\gamma_n^{2}\|w_n\|\right).
\end{align*}
There exists a point $\hat x_n\in\hat {H}^{-1}_{0}(0)\cap S$ such that
$$
\|\bx+\gamma_nu+\dfrac{1}{2}\gamma_n^2x_n-\hat x_n\| \leq\mu\left(\dfrac{\theta}{2}\gamma_n^{2}\|w_n\|\right) +\ga_n^3.
$$
By setting $x_n':=(\frac{1}{2}\gamma_n^2)^{-1}(\hat x_n-\bar x-\gamma_nu)$, one has $\bx+\gamma_nu+\dfrac{1}{2}\gamma_n^2x_n'
\in \mathcal{H}^{-1}_{0}(0)\cap S$ and
$$\|x_n-x_n'\|\leq \dfrac{\mu\left(\dfrac{\theta}{2}\gamma_n^{2}\|w_n\|\right)} {\dfrac{1}{2}\gamma_n^2}+2\ga_n\to 0
\quad\mbox{as}\quad n\to\infty.$$
Hence, $x_n'\to x$ as $n\to\infty$.
As $z\in D_{l}^{2}G_{+}(\bx,\bz,u,k)(x)$, there exists $z_n \to z$ such that $\bz+\gamma_nk+\dfrac{1}{2}\gamma_n^{2}z_n \in
G(\bx+\gamma_nu+\dfrac{1}{2}\gamma_n^{2}x_n')+D$ for large $n$.
Moreover, as $z\in IT^{2}(-D,\bz,k)$, it holds $\bar z+\gamma_nk+\dfrac{1}{2}\gamma_n^{2}z_n \in -D$
for large $n$.
Hence,
$(G(\bx+\gamma_nu+\dfrac{1}{2}\gamma_n^{2}x_n')+D)\cap (-D)\ne \emptyset$ and, as $D$ is a convex cone, $G(\bx+\gamma_nu+\dfrac{1}{2}\gamma_n^{2}x_n')\cap (-D) \ne \emptyset$ for large $n$.
Thus, $\bx+\gamma_nu+\dfrac{1}{2}\gamma_n^{2}x_n' \in \Omega$ for large $n$, i.e., $x \in T^2(\Omega, \bx,u)$.

\underline{Claim 3}.
$\Delta_{(u,v,k)}(IT^{2}(S,\bx,u))\bigcap \big ((-\cone(Q+v))\times IT^{2}(-D,\bz,k) \times \{0\}\big)=\emptyset$.
\smallskip

Suppose to the contrary the existence of $x\in IT^{2}(S,\bx,u)$, $y\in-\cone(Q+v)$ and $z\in IT^{2}(-D,\bz,k)$ such that
$
(y,z,0)\in \Delta_{(u,v,k)}(x).
$
Then, by Claim 2, $x\in T^2(\Omega,\bx,u)$.
We arrive at a contradiction with Claim 1.
\smallskip

\underline{Claim 4}.
\emph{There exist multipliers $(v^*,k^*,w^* )\in Q^*\times N(-D,\bar z)\times W^* \setminus \{(0,0,0)\}$ such that $\langle v^* ,v \rangle=\langle k^* ,k \rangle =0$ and
\eqref{5.1} holds true.}
\smallskip

If $IT^{2}(-D,\bz,k)=\emptyset$, then $A^{2}(-D,\bz,k)=\emptyset$ and \eqref{5.1} holds true trivially.
Let $IT^{2}(-D,\bz,k)\ne\emptyset$.
The standard separation theorem applied to the two convex sets in Claim 3 yields the existence of multipliers $(v^*,k^*,w^*)\in Y^*\times Z^*\times W^* \setminus \{(0,0,0)\}$ such that
\begin{gather}\label{5.8}
\langle v^* ,y \rangle+ \langle k^* ,z \rangle+\langle w^*,w \rangle \geq
\langle v^* ,q\rangle+ \langle k^*,d\rangle
\end{gather}
for all $x \in IT^{2}(S,\bx,u)$, $(y,z,w)\in\Delta_{(u,v,k)}(x)$, $q\in-\cone(Q+v)$, and all $d \in IT^{2}(-D,\bz,k)$.
For any fixed admissible $x,y,z,w$ and $d$ and any $q\in\cone(Q+v)$ and $t>0$, one has $-tq\in-\cone(Q+v)$, and consequently,
$$
\langle v^*,q\rangle\ge\lim_{t\to\infty}\frac{\langle v^*,y\rangle+ \langle k^*,z\rangle+\langle w^*,w \rangle-\langle k^*,d\rangle}{t}=0.
$$
Hence,
\begin{gather}\label{5.7}
\langle v^*,q \rangle\ge0
\quad\mbox{for all}\quad
q\in\cone(Q+v),
\end{gather}
and consequently, taking into account the second property in \eqref{ys0}, inequality \eqref{5.8} implies \eqref{5.1}.

Since $Q$ is a cone, by the same argument, it follows from \eqref{5.7} that $v^*\in Q^*$.
As $v\in-\bd Q$, we also have $\langle v^*,v\rangle=0$.
Using \eqref{5.7} and property \eqref{ys} of the adjacent set, we obtain from \eqref{5.8} that
\begin{gather*}
\langle v^* ,y \rangle+ \langle k^* ,z \rangle+\langle w^*,w \rangle \geq
\langle k^* ,d\rangle+ \langle k^*,d'\rangle
\end{gather*}
for all $x \in IT^{2}(S,\bx,u)$, $(y,z,w)\in\Delta_{(u,v,k)}(x)$, $d \in A^{2}(-D,\bz,k)$, and all $d' \in T(T(-D,\bz),k)$.
Using the fact that $T(T(-D,\bz),k)$ is a cone, we conclude as before that $k^* \in -(T(T(-D,\bz),k))^*$, and consequently, $k^*\in N(-D,\bz)$.
As $k\in T(-D,\bz)$, we also have $\langle k^*,k\rangle=0$.
\smallskip

\underline{Claim 5}.
\emph{Under the constraint qualification \eqref{5.2}, $v^*$ in \eqref{5.1} is nonzero.}
\smallskip

Suppose that $v^*=0$.
Then, $(k^*,w^*)\neq(0,0)$ and \eqref{5.1} gives
\begin{align}\label{5.9}
\langle k^* ,z \rangle+\langle w^* ,w \rangle\geq \langle k^*, d \rangle
\end{align}
for all  $x \in IT^{2}(S,\bx,u)$, $z\in D_{l}^{2}G_{+}(\bx,\bz,u,k)(x)$,
$w\in D^{2}H(\bx,0,u,0)(x)$ and $d \in A^{2}(-D,\bz,k)$.
Take arbitrarily $(z',w')\in Z\times W$.
By virtue of \eqref{5.2}, there are $x\in IT^2(S,\bx,u)$, $z\in D_l^{2}G_{+}(\bx,\bz,u,k)(x)$, $w\in D^{2}H(\bx,0,u,0)(x)$, $d\in A^2(-D,\bar z,k)$, $d'\in D$ and
$\gamma_1, \gamma_2>0$
such that $(z',w')=\gamma_1 (z-d,w) +(\gamma_2(d'+\bz),0)$.
Since $k^* \in N(-D,\bz)$, one has $\langle k^*, d'\rangle \geq 0$ and $\langle k^*, \bz\rangle=0$.
Hence, using \eqref{5.9},
\begin{align*}
\langle k^*, z' \rangle+\langle w^*, w'\rangle &=\gamma_1\langle k^*,z-d\rangle +\gamma_2\langle k^*,d'+\bz\rangle +\gamma_1 \langle w^*,w\rangle
\\
&= \gamma_1 (\langle k^*, z\rangle +\langle w^*, w\rangle -\langle k^*, d\rangle)+\gamma_2\langle k^*, d'+\bz\rangle
\\
&\ge\gamma_2\langle k^*, d'+\bz\rangle\ge0.
\end{align*}
As $(z',w')\in Z\times W$ is arbitrary, we have $(k^*,w^*)=(0,0)$, a contradiction.
\end{proof}

\begin{remark}
{\rm 1}.
The requirements on the extension mapping $\mathcal{H}$ formulated before Definition~\ref{D5.1} are satisfied, e.g., if
$$e(\mathcal{H}(x,t),H(x)):=\sup_{h\in\mathcal{H}(x,t)}d(h,H(x))\leq \alpha t^k$$
for some $\alpha>0$, $k\geq 1$ and all $(x,t)$ in a neighbourhood of $(\bx,0)$.

{\rm 2}. The lower estimate for $T^2(\Omega,\bx,u)$ in Claim 2  and its proof presented above  are valid for any feasible triple $(\bx,\by,\bz)$ and any $u \in X$ with $0\in DH(\bx,0)(u)$ and $k \in D_{l}G_{+}(\bx,\bz)(u)$.
This estimate can be of importance beyond Theorem~\ref{T5.2}.

{\rm 3}. In the proof of Theorem~\ref{T5.2} (see Claim 2), one can employ weaker regularity properties of the extension mapping $\mathcal{H}$ than the one given in Definition~\ref{D5.1}.
Firstly, it is sufficient to require the inequality in \eqref{D5.1-1} to hold only at the fixed point $w=0$.
This important property known as \emph{metric subregularity} can be treated in the abstract setting of the current article and is going to make the topic of subsequent research.
Moreover, only points of the form $\bx+\gamma_nu+\dfrac{1}{2}\gamma_n^2x_n$ are involved in the proof.
Hence, a development of our regularity model corresponding to \emph{directional metric subregularity} is on the agenda.
Such an extension is going to properly improve \cite[Theorem~3.1]{KhaTun15}.

{\rm 4}. Following \cite{KhaTun15}, one can improve Theorem~\ref{T5.2} by relaxing the restrictive assumption of nonemptyness of the interior of the set $\Delta_{(u,v,k)}(IT^{2}(S,\bar x,u))$.

{\rm 5}. It is possible to develop multiplier rules similar to the one in Theorem~\ref{T5.2} in terms of other types of generalized derivatives, for instance asymptotic derivatives, instead of the contingent-type ones.
Such rules may be useful when the contingent-type derivatives do not exist in a particular problem under consideration.
\end{remark}

\section{Concluding remarks}\label{ConRem}

This article considers a general regularity model for a set-valued mapping $F:X\times\R_+\rightrightarrows Y$, where $X$ and $Y$ are metric spaces.
We demonstrate that the classical approach going back to Banach, Schauder, Lyusternik and Graves and based on iteration procedures still possesses potential.
In particular, we show that the \emph{Induction theorem} \cite[Theorem~1]{Kha86}, which was used as the main tool when proving the other results in \cite{Kha86}, implies also all the main results in the subsequent articles \cite{Kha88,Kha89} and
can serve as a substitution of the Ekeland variational principle when establishing other regularity criteria.
Furthermore, the latter classical result can also be established as a consequence of the Induction theorem.

This research prompts a list of questions and problems which should be taken care of.

1) ``On a set'' nonlinear regularity, considered in Section~\ref{MET_SEC} and interpreted there as a direct analogue of metric regularity in the conventional setting, is in fact a general model which covers also relaxed versions of regularity like sub- and semi-regularity.

2) The particular case of ``power nonlinearities'', i.e., the case when functional modulus $\mu$ is of the type $\mu(t)=\lambda t^k$ with $0<k\le1$, should be treated explicitly.

3) Theorem \ref{T2.3} illustrates the usage of the Induction theorem as a substitution for the Ekeland variational principle when establishing regularity criteria like Theorem~\ref{T6.4}.
In the last theorem which is an (indirect) consequence of Theorem \ref{T2.3}, the mapping is assumed upper semicontinuous.
This assumption can be relaxed with the help of a slightly more advanced version of Theorem \ref{T2.3}.

4) The regularity model studied in this article is illustrated in Section~\ref{S5} by an application to second-order necessary optimality conditions for a nonsmooth set-valued optimization problem with mixed constraints.
Other classes of optimization problems can be handled along the same lines using also other types of generalized derivatives.
The relaxed versions of regularity mentioned in item 1 above are going to be useful in this context.

\section*{Acknowledgements}
We thank the two anonymous referees for the very careful reading of the manuscript and many helpful comments and suggestions which helped us substantially improve the presentation.

\end{document}